\documentclass[10pt]{article}
\usepackage{bbm}
\usepackage{amsmath}
\usepackage{amsfonts}
\usepackage{amssymb}
\usepackage{multirow}
\usepackage{mathrsfs}
\usepackage{cite}
\usepackage{ascii}
\usepackage{bbding}
\usepackage{pifont}
\usepackage{ifsym}
\usepackage{wasysym}
\usepackage{upgreek}
\usepackage{color}
\usepackage{graphicx}
\usepackage{eso-pic}
\usepackage{epstopdf}
\usepackage{setspace}
\usepackage{enumerate}
\usepackage{tikz}
\usepackage{colortbl}
\usepackage{booktabs}
\usepackage{makecell}
\usepackage[figuresright]{rotating}
\usepackage{longtable}
\usepackage{caption}
\usepackage{appendix}
\usepackage{algorithm}
\usepackage{algpseudocode}

\usepackage[hang]{footmisc}

\AtBeginDocument{
 \footnotesize\setlength{\footnotemargin}{1em}\normalsize
 \edef\hangfootparindent{\the\parindent}
 
}

\renewcommand{\Box}{\framebox{\rule{0.3em}{0.0em}}}
\newtheorem{thm}{Theorem}[section]
\newtheorem{lema}{Lemma}[section]

\newtheorem{ex}{Example}[section]
\newtheorem{rem}{Remark}[section]
\newtheorem{defi}{Definition}[section]


\newenvironment{proof}{\noindent{\bf Proof. }}{\hfill $\Box$\medskip}
\renewcommand{\Box}{\hfill\rule{2.3mm}{2.3mm}}

\renewcommand{\Box}{\framebox{}}
\numberwithin{equation}{section}
\title{Solving bilevel programs based on lower-level Mond-Weir duality\thanks{This work was supported in part by NSFC (Nos. 12071280, 11901380).}
}
\author{Yu-Wei Li\thanks{\baselineskip 9pt School of Management, Shanghai University, Shanghai 200444, China. E-mail: yuwei\_li@shu.edu.cn.}, \ Gui-Hua Lin\thanks{\baselineskip 9.5pt Corresponding author. School of Management, Shanghai University, Shanghai 200444, China. Email: guihualin@shu.edu.cn.}, \ Xide Zhu\thanks{\baselineskip 9pt
School of Management, Shanghai University, Shanghai 200444, China. E-mail: xidezhu@shu.edu.cn.}
}

\date{\bigskip 
Submitted: March 2023; Revised: October 2023\\[1mm]
Published in INFORMS Journal on Computing: February 2024\\[1.5mm]
https://doi.org/10.1287/ijoc.2023.0108}

\begin{document}
\maketitle

\baselineskip 18pt

\noindent{\bf Abstract.}
This paper focuses on developing effective algorithms for solving bilevel program. The most popular approach is to replace the lower-level problem by its Karush-Kuhn-Tucker conditions to generate a mathematical program with complementarity constraints (MPCC). However, MPCC does not satisfy the Mangasarian-Fromovitz constraint qualification (MFCQ) at any feasible point. In this paper, inspired by a recent work using the lower-level Wolfe duality (WDP), we apply the lower-level Mond-Weir duality to present a new reformulation, called MDP, for bilevel program. It is shown that, under mild assumptions, they are equivalent in globally or locally optimal sense. An example is given to show that, different from MPCC, MDP may satisfy the MFCQ at its feasible points. Relations among MDP, WDP, and MPCC are investigated. On this basis, we extend the MDP reformulation to present another new reformulation (called eMDP), which has similar properties to MDP. Furthermore, in order to compare two new reformulations with the MPCC and WDP approaches, we design a procedure to generate 150 tested problems randomly and comprehensive numerical experiments show that MDP has quite evident advantages over MPCC and WDP in terms of feasibility to the original bilevel programs, success efficiency, and average CPU time,while eMDP is far superior to all other three reformulations.

\vspace{4pt}\noindent{\bf Keywords.}
Bilevel program, Mond-Weir duality, Wolfe duality, MPCC, constraint qualification.

\vspace{4pt}\noindent{\bf 2010 Mathematics Subject Classification.} \ 90C30, 90C33, 90C46.
\baselineskip 18pt

\section{Introduction}\label{intro}
In this paper, we focus on solving the bilevel program
\begin{eqnarray*}
{\rm(BP)}\qquad\min&&F(x, y)\\
\mbox{s.t.}&&(x,y)\in \Omega,~y \in {S}(x),
\end{eqnarray*}
where $x\in\mathbb{R}^n$ and $y\in\mathbb{R}^m$ represent the upper-level and lower-level variables, respectively, and $S(x)$ denotes the optimal solution set of the lower-level parameterized problem
\begin{eqnarray*}
{(\mathrm{P}_x)}\qquad\min\limits_{y}&&f(x, y)\nonumber\\
\mbox{s.t.}&&g(x, y) \le 0, ~h(x, y)=0.\nonumber
\end{eqnarray*}
Throughout, we denote by
$X:=\{x\in \mathbb{R}^n: \exists~ y\in \mathbb{R}^m~{\rm s.t.}~(x,y)\in\Omega\}$ and $Y(x) := \{y \in \mathbb{R}^{m} : g(x, y) \le 0, ~h(x, y) = 0\}.$
We assume that the upper-level objective function $F: \mathbb{R}^{n+m}\rightarrow\mathbb{R}$ is continuously differentiable, the lower-level functions $f: \mathbb{R}^{n+m} \rightarrow\mathbb{R}$, $g$: $\mathbb{R}^{n+m} \rightarrow\mathbb{R}^{p}$, $h$: $\mathbb{R}^{n+m} \rightarrow\mathbb{R}^{q}$ are all twice continuously differentiable, and ${S}(x) \neq \emptyset$ for each $x\in X$.

Bilevel program was first introduced in market economy by Stackelberg \cite{Stackelberg1952theory} and hence is also called a Stackelberg game or leader-follower game model in game theory. Systematic studies on bilevel program started from the work of \cite{Bracken1973mathematical}.
Nowadays, bilevel program has been applied to many fields such as industrial engineering, supply chain management, economic planning, machine learning, and so on \cite{Kunapuli2008classification, caprara2016bilevel, Lachhwani2018bilevel, ye2023difference, liu2023hierarchical, liu2023value}. For more details about developments on bilevel program, we refer the readers to \cite{Bard1998practical, Colson2007overview, Dempe2013bilevel, zeng2020practical, kleinert2021computing, byeon2022benders, lu2023penalty} and the references therein.

Bilevel program is generally difficult to solve due to its hierarchical structure. In fact, even for the linear case, it has been shown to be strongly NP-hard by Hansen \cite{Hansen1992new}.
So far, the most popular approach in dealing with BP is to make use of the Karush-Kuhn-Tucker (KKT) conditions of the lower-level problem to transform it into the mathematical program with complementarity constraints
\begin{eqnarray*}
{\rm(MPCC)}\qquad\min && F(x,y) \\
\mbox{s.t.}&& (x,y)\in \Omega,~h(x, y)=0,\\
&&u\geq0,~g(x, y)\le 0,~ u^Tg(x, y)=0,\\
&&\nabla_y f(x, y) + \nabla_y g(x, y) u+ \nabla_y h(x, y)v=0.
\end{eqnarray*}
If the lower-level problem $\mathrm{P}_x$ is convex in $y$ and satisfies some constraint qualifications, BP and MPCC are equivalent in globally optimal sense.
However, MPCC does not satisfy the Mangasarian-Fromovitz constraint qualification (MFCQ) at any feasible point so that the well-developed optimization algorithms in nonlinear programming may not be stable in solving MPCC.
There have been proposed a number of approximation algorithms such as relaxation and smoothing algorithms, penalty function algorithms, interior point algorithms, implicit programming algorithms, active-set identification algorithms, constrained equation algorithms, and nonsmooth algorithms for solving MPCC. See, e.g., \cite{Luo1996mathematical, Fukushima1998globally, Scholtes2001convergence, Lin2006hybrid, Lin2009solving, izmailov2012semismooth, Hoheisel2013theoretical, Guo2015solving} and the references therein for more details about MPCC. Another approach in dealing with bilevel program is based on the lower-level optimal value function and some optimality conditions and constraint qualifications were derived through this way \cite{Ye1995optimality, Ye2010new, Dempe2013bilevel}.
Although the lower-level optimal value function does not have analytic expressions in general and so it can not be solved directly by the popular optimization algorithms, some approximation algorithms were presented along this approach \cite{Lin2014solving, Xu2014smoothing, Ma2021combined}.

Recently, \cite{Li2022novel} suggested a new approach to solve bilevel program based on the lower-level Wolfe duality. Under some kind of convexity and the Guignard constraint qualification, it is shown that BP is equivalent to the single-level optimization problem
\begin{eqnarray*}
{\rm (WDP)}\qquad\min &&F(x, y)  \\
\mbox{s.t.}&& (x,y)\in \Omega,~g(x, y) \le 0, ~h(x, y)=0,\\
&&f(x, y) - f(x, z) - u^{T} g(x, z) - v^{T} h(x, z) \leq 0, \\
&&\nabla_z f(x, z) + \nabla_z g(x, z) u + \nabla_z h(x, z) v = 0, ~u\geq 0.	
\end{eqnarray*}
\noindent More interestingly, it was shown by example that, different from MPCC that fails to satisfy the MFCQ at any feasible point, WDP may satisfy the MFCQ at its feasible point. A relaxation method based on WDP was presented and numerical experiments indicate that, compared with MPCC, this approach is more efficient. See \cite{Li2022novel} for more details.

Inspired by the above work, we consider another approach for bilevel program based on the lower-level Mond-Weir duality. It will be shown in the subsequent sections that, under some conditions different from that assumed for WDP, BP is equivalent to the single-level optimization problem
\begin{eqnarray*}
{\rm (MDP)}\qquad\min&&F(x, y) \nonumber \\
\mbox{\rm s.t.}&&(x,y)\in \Omega,~g(x, y) \le 0, ~h(x, y)=0,\\
&&f(x, y) - f(x, z)\leq 0,~ u^T g(x, z) + v^T h(x,z)\geq0,\nonumber\\
&&\nabla_z f(x, z) + \nabla_z g(x, z) u + \nabla_z h(x, z) v = 0, ~u\geq 0.\nonumber
\end{eqnarray*}
We use an example to show the necessity to study this new approach. In fact, the following bilevel program is equivalent to the new MDP reformulation, but not to the WDP reformulation.

\begin{ex}\label{ex_WDP_MDP}\rm
Consider the bilevel program
\begin{eqnarray}\label{ex_WDP_MDP_1}
\min&&-x-y\\
\mbox{\rm s.t.}&&x\leq 1,~y \in \arg\min\limits_{y}\{y^3+y:y\geq x\}.\nonumber	
\end{eqnarray}
It is obvious that $(1,1)$ is a unique globally optimal solution to \eqref{ex_WDP_MDP_1} with the optimal value $-2$. For this problem, the WDP reformulation is
\begin{eqnarray}\label{ex_WDP_MDP_2}
\min && -x-y \nonumber\\
\mbox{\rm s.t.}&& x\leq 1,~ y\geq x,~3z^2+1-u=0, ~u\geq 0, \\
&&y^3+y- z^3-z+u(z-x)\leq0.\nonumber	
\end{eqnarray}
Consider the feasible point sequence $(x^k,y^k,z^k,u^k)=(0,k,-k,3k^2+1)$, $k=1,2,\cdots$. Letting $k\rightarrow +\infty$, we know that the optimal value of \eqref{ex_WDP_MDP_2} is $-\infty$, which means that \eqref{ex_WDP_MDP_1} is not equivalent to \eqref{ex_WDP_MDP_2} in globally optimal sense.

On the other hand, the MDP reformulation for \eqref{ex_WDP_MDP_1} is
\begin{eqnarray}\label{ex_WDP_MDP_3}
\min && -x-y \nonumber\\
\mbox{\rm s.t.}&& x\leq 1,~ y\geq x,~3z^2+1-u=0, ~u\geq 0,\\
&&y^3+y\leq z^3+z,~u(z-x)\leq 0.\nonumber	
\end{eqnarray}
The constraints $3z^2+1-u=0$ and $-u(z-x)\geq 0$ imply $z\leq x$. This, together with $y^3+y\leq z^3+z$ and $y\geq x$, implies $x=y=z$ by the strict monotonicity of the function $t^3+t$. It is easy to see that $(1,1,1,4)$ is the unique globally optimal solution of \eqref{ex_WDP_MDP_3} with the optimal value $-2$. This means that \eqref{ex_WDP_MDP_1} is equivalent to \eqref{ex_WDP_MDP_3} in globally optimal sense.
\end{ex}

The above example reveals that, compared with the WDP reformulation, the new MDP reformulation may be applied to different cases. Indeed,  the equivalent conditions assumed for MDP are different from that assumed for WDP in \cite{Li2022novel}. Moreover, the MDP reformulation  has some other characteristics. For example, its feasible region is smaller than the WDP reformulation and, when the lower-level objective is composed of a strictly monotone one-dimensional function $f(t)$ with $t=r(x,y)$, the constraint $f(x, y) - f(x, z)\leq 0$ may be simplified, but the WDP reformulation does not have this merit. The numerical experiments reported in Section \ref{numerical} indicate that the new MDP reformulation indeed has visible advantages over the WDP reformulation.

The paper is organized as follows.
In Section \ref{Preliminaries}, we review some constraint qualifications and some useful results about the Mond-Weir duality for constrained optimization problems.
In Section \ref{Reformulation-MDP}, we show the equivalence between BP and MDP in both globally and locally optimal senses. In Section \ref{conterex}, we show by an example that the MDP reformulation may satisfy the MFCQ at its feasible point, which is similar to WDP, but different from MPCC. In Section \ref{Comparison of three reformulations}, we investigate the relations among the stationary points of MDP, MPCC, and WDP, respectively.
In Section \ref{extension}, we extend the MDP reformulation to get a new reformulation (called eMDP). In particular, the weak and strong duality theorems related to eMDP and the convergence analysis for relaxation schemes are given in Appendix.
In Section \ref{numerical}, we focus on the numerical efficiency of the new approaches by comparing with MPCC and WDP. Our numerical examples include $150$ randomly generated linear bilevel programs and numerical experiments indicate that MDP has quite evident advantages over MPCC and WDP, while eMDP is far superior to all other three reformulations.
In Section \ref{conclusion}, we make some concluding remarks.

\section{Preliminaries}
\label{Preliminaries}
In this section, we recall some concepts and results useful in the subsequent sections. Consider the constrained optimization problem
\begin{eqnarray*}
{\rm (P)}\qquad\min&& f(y)\\
\mbox{s.t.}&& g(y) \le 0, ~h(y) = 0,
\end{eqnarray*}
where $f: \mathbb{R}^{m}\rightarrow\mathbb{R}$, $g$: $\mathbb{R}^{m}\rightarrow\mathbb{R}^{p}$, and $h$: $\mathbb{R}^{m}\rightarrow\mathbb{R}^{q}$.
Denote by $Y$ the feasible region of P.

\subsection{Constraint qualifications}
In this subsection, we recall some constraint qualifications used later on. Given a point $\bar{y} \in Y$, we define an index set as $I_{\bar{g}}:=\{ i : g_i(\bar{y})=0\}$.
The tangent and linearized cones of the set ${Y}$ at $\bar{y}$ are defined as
\begin{eqnarray*}
\mathcal{T}_{Y}(\bar{y}) := \left\{d \in \mathbb{R}^{m} :
		\exists\{y^{k}\} \subseteq {Y},~ \exists\{t_{k}\} \downarrow 0
		~{\rm such~ that}~ y^{k} \rightarrow \bar{y} ~{\rm and}~ \frac{y^{k}-\bar{y}}{t_{k}} \rightarrow d \right\}	
\end{eqnarray*}
and
\begin{eqnarray*}	
	\mathcal{L}_{Y}(\bar{y}) := \Big\{ d \in \mathbb{R}^{m} : \nabla g_{i}(\bar{y})^T d \leq 0~i \in I_{\bar{g}}, ~~\nabla h_{i}(\bar{y})^T d = 0~i=1, \cdots, q \Big\},
\end{eqnarray*}
respectively. For a given cone $D\subset \mathbb{R}^{m}$, its polar cone is defined as
	$D^\circ := \big\{d \in \mathbb{R}^{m} :  d^T y \leq0~\mbox{for~all}~y\in D  \big\}.$	

\begin{defi}
We say that problem P satisfies
		\begin{itemize}		
			\item
			Mangasarian-Fromovitz constraint qualification (MFCQ) at $\bar{y}\in {Y}$ if the set of gradients $\{\nabla h_j(\bar{y}) : j=1,\cdots, q\}$ is linearly independent and there exists $ d\in\mathbb{R}^m$ such that $\nabla g_i(\bar{y})^T d <0$ $(i\in I_{\bar{g}})$ and $\nabla h_j(\bar{y})^T d =0~(j=1,\cdots, q)$;		
			\item
			Guignard constraint qualification at $\bar{y}\in {Y}$ if $ \mathcal{L}_{Y}(\bar{y})^\circ = \mathcal{T}_{Y}(\bar{y})^\circ $.
		\end{itemize}
\end{defi}

\subsection{Mond-Weir duality for constrained optimization problems}
The Mond-Weir dual problem given in \cite{Mond1981generalized} for P is
\begin{eqnarray*}
{\rm (MD)}\qquad\max &&f(z)\\
\mbox{s.t.}&& \nabla_z  L(z, u, v)=0,~u^{T} g(z) + v^{T} h(z)\geq 0,~u\geq 0,	
\end{eqnarray*}
where $L(z, u, v) := f(z) + u^{T} g(z) + v^{T} h(z)$ represents the Lagrangian function. We denote by $M$ the feasible region of MD, and for any fixed $(u,v)$, we let
\begin{eqnarray*}
\mathcal{Z}(u,v):=\{z\in \mathbb{R}^m: \nabla_z L(z, u, v)=0 ,~u^{T} g(z) + v^{T} h(z)\geq 0\}.
\end{eqnarray*}
Note that, due to the existence of the terms $u^{T} g$ and $v^{T} h$, the Mond-Weir dual problem is usually a nonconvex optimization problem, even though the primal problem is convex.

To introduce the duality theorems for Mond-Weir duality, we need the following well-known generalized convexity concepts.

\begin{defi}\rm \label{convexity}
Let $f:\mathbb{R}^m \to \mathbb{R}$ be a differentiable function. $f$ is called pseudoconvex on $\mathcal{Y}\subseteq\mathbb{R}^m$ if $(y_1-y_2)^T\nabla f(y_2) \geq 0$ implies $f(y_1) \geq f(y_2)$ for any $y_1, y_2\in\mathcal{Y}$. $f$ is called quasiconvex on $\mathcal{Y}\subseteq\mathbb{R}^m$ if $f(y_1) \leq f(y_2)$ implies $(y_1-y_2)^T\nabla f(y_2)\leq 0$ for any $y_1, y_2\in\mathcal{Y}$.
\end{defi}

From Theorems 3 and 4 in \cite{Egudo1986duality}, we have the following duality theorems.

\begin{thm}\label{weak duality}\emph{\textbf{(weak Mond-Weir duality)}}
Suppose that $f$ is pseudoconvex and $u^{T} g + v^{T} h$ is quasiconvex over $Y\cup \mathcal{Z}(u,v)$ for any $u\in \mathbb{R}^p_+$ and $v\in \mathbb{R}^q$. Then, the optimal value of the primal problem is no less than its Mond-Weir dual problem, i.e.,
\begin{eqnarray*}
\min\limits_{y\in Y} f(y) \geq \max\limits_{(z, u, v)\in {M}} f(z).
\end{eqnarray*}
\end{thm}

\begin{thm}\label{strong duality}\emph{\textbf{(strong Mond-Weir duality)}}
Suppose that $f$ is pseudoconvex and $u^{T} g + v^{T} h$ is quasiconvex over $Y\cup \mathcal{Z}(u,v)$ for any $u\in \mathbb{R}^p_+$, $v\in \mathbb{R}^q$, and {\rm P} satisfies the Guignard constraint qualification at some solution $\bar{y}$. Then, there exists an optimal solution $(\bar{z},\bar{u},\bar{v})$ of {\rm MD} such that
\begin{eqnarray*}
\min\limits_{y\in Y} f(y)  =f(\bar{y})=f(\bar{z}) =\max\limits_{(z, u, v)\in {M}} f(z).
\end{eqnarray*}
\end{thm}

Note that, in Theorem \ref{strong duality}, the dual optimal solution $\bar{z}$ may not solve the primal problem P, unless there are some restrictions on the functions involved. See the following converse duality theorem given in \cite[Theorem 7]{Egudo1986duality}

\begin{thm}\label{converse duality}\emph{\textbf{(converse duality)}}
Let $\bar{y}$ solve {\rm P} and $(\bar{z},\bar{u},\bar{v})$ solve {\rm MD}.	Suppose that $f$ is pseudoconvex and $u^{T} g + v^{T} h$ is quasiconvex over $Y\cup \mathcal{Z}(u,v)$ for any $u\in \mathbb{R}^p_+$ and $v\in \mathbb{R}^q$. If $f$ is strictly pseudoconvex near $\bar{z}$ and {\rm P} satisfies the Guignard constraint qualification at $\bar{y}$, then $\bar{z}=\bar{y}$.
\end{thm}

\section{New reformulation based on lower-level Mond-Weir duality }
\label{Reformulation-MDP}
In this section, we investigate the relations between BP and MDP. For convenience, we write the Mond-Weir dual problem of the lower-level problem $\mathrm{P}_x$ as
\begin{eqnarray*}\label{MDx}
{(\mathrm{MD}_x)}\qquad\max\limits_{z, u, v}&&f(x,z) \nonumber\\
\mbox{s.t.}&& \nabla_z L(x, z, u, v)=0,~u^{T} g(x, z) + v^{T} h(x, z)\geq 0,~u\geq 0.
\end{eqnarray*}
Denote by $M(x)$ the feasible region of $\mathrm{MD}_x$ and, for any fixed $(x,u,v)$, let
\begin{eqnarray*}
\mathcal{Z}(x,u,v):=\{z\in \mathbb{R}^m: \nabla_z L(x,z, u, v)=0,~u^{T} g(x, z) + v^{T} h(x, z)\geq 0\}.
\end{eqnarray*}
We first show the equivalence between BP and MDP in globally optimal sense.

\begin{thm}\label{MDP-th}
Suppose that, for any given $x\in X$, $u\in \mathbb{R}^p_+$, and $v\in \mathbb{R}^q$, $f(x,\cdot)$ is pseudoconvex and $u^{T} g(x, \cdot) + v^{T} h(x,\cdot)$ is quasiconvex over $Y(x)\cup \mathcal{Z}(x,u,v)$, and $\mathrm{P}_x$ satisfies the Guignard constraint qualification at some solution $y_x \in {S(x)}$. Then, if $(\bar{x},\bar{y})$ is an optimal solution to {\rm BP}, there exists $(\bar{z},\bar{u},\bar{v})$ such that $(\bar{x},\bar{y},\bar{z},\bar{u},\bar{v})$ is an optimal solution to {\rm MDP}. Conversely, if $(\bar{x},\bar{y},\bar{z},\bar{u},\bar{v})$ is an optimal solution to {\rm MDP}, then $(\bar{x},\bar{y})$ is an optimal solution to {\rm BP}.
\end{thm}

\begin{proof}
By the assumption that ${S}(x) \neq \emptyset$ for each $x\in X$, BP is equivalent to
\begin{eqnarray*}\label{VP}
\min&&F(x,y)\nonumber\\
\mbox{s.t.}&& (x,y)\in \Omega,~ y\in Y(x),\\
&&f(x, y) \leq \min\limits_{y \in Y(x)} f(x, y)\nonumber
\end{eqnarray*}
in globally optimal sense. On the other hand, the assumptions and Theorem \ref{strong duality} imply that, for each $x\in X$, the strong Mond-Weir duality holds between $\mathrm{P}_x$ and $\mathrm{MD}_x$, i.e., there exist $y_x\in Y(x)$ and $(z_x,u_x,v_x)\in M(x)$ such that
\begin{eqnarray*}	
\min\limits_{y\in Y(x)} f(x,y)=f(x,y_x)=f(x, z_x)=\max\limits_{(z, u, v)\in M(x)} f(x, z).
\end{eqnarray*}	
Thus, BP is further equivalent to
\begin{eqnarray}\label{MDVP}
\min&&F(x,y)\nonumber\\
\mbox{s.t.} && (x,y)\in \Omega,~ y\in Y(x),\\
&&f(x, y) \leq \max\limits_{(z, u, v) \in M(x)} f(x, z)\nonumber
\end{eqnarray}
in globally optimal sense. Next, we show that \eqref{MDVP} is equivalent to {\rm MDP} in globally optimal sense.
In fact, if $(\bar{x}, \bar{y})$ is a globally optimal solution to \eqref{MDVP}, it follows that
\begin{eqnarray}\label{MDP_2}
F(\bar{x}, \bar{y}) \leq F(x,y)
\end{eqnarray}
for any feasible point $(x,y)$ to \eqref{MDVP} and, by the feasibility of $(\bar{x}, \bar{y})$ to \eqref{MDVP},
\begin{eqnarray*}
f(\bar{x}, \bar{y}) \leq \max\limits_{(z, u, v)\in M(\bar{x})} f(\bar{x}, z)=f(\bar{x}, z_{\bar{x}}).
\end{eqnarray*}
This, together with $(z_{\bar{x}}, u_{\bar{x}}, v_{\bar{x}})\in M(\bar{x})$, indicates that $(\bar{x}, \bar{y}, z_{\bar{x}}, u_{\bar{x}}, v_{\bar{x}})$ is feasible to {\rm MDP}.
For any feasible point $(x, y, z, u, v)$ to {\rm MDP}, since
\begin{eqnarray*}
f(x, y) \leq f(x, z)\leq\max\limits_{(z', u', v') \in M(x)} f(x, z'),
\end{eqnarray*}
where the second inequality follows from $(z, u, v) \in M(x)$, we have that $(x, y)$ is feasible to \eqref{MDVP} and hence \eqref{MDP_2} holds. By the arbitrariness of $(x, y, z, u, v)$,  $(\bar{x}, \bar{y}, z_{\bar{x}}, u_{\bar{x}}, v_{\bar{x}})$ is a globally optimal solution to {\rm MDP}.

Conversely, if $(\bar{x}, \bar{y}, \bar{z}, {\bar{u}}, {\bar{v}})$ is a globally optimal solution to {\rm MDP}, we have $\nabla_z L(\bar{x}, \bar{z}, \bar{u}, \bar{v})=0, {\bar{v}}^{T} g(\bar{x}, \bar{z}) + {\bar{v}}^{T} h(\bar{x}, \bar{z})\geq0, \bar{u}\ge 0$ by $(\bar{z}, {\bar{u}}, {\bar{v}})\in M(\bar{x})$ and
\begin{eqnarray*}
f(\bar{x}, \bar{y}) \leq f(\bar{x},\bar{z})\le \max\limits_{(z, u, v)\in M(\bar{x})} f(\bar{x}, z),
\end{eqnarray*}
which means that $(\bar{x}, \bar{y})$ is feasible to \eqref{MDVP}. For any feasible point $(x, y)$ to \eqref{MDVP}, there must exist $(z_x,u_x,v_x)\in M(x)$ such that
\begin{eqnarray*}
f(x, y) \leq \max\limits_{(z, u, v) \in M(x)} f(x, z)= f(x,z_x)
\end{eqnarray*}
and so $(x,y,z_x,u_x,v_x)$ is feasible to {\rm MDP}. By the optimality of $(\bar{x}, \bar{y}, \bar{z}, \bar{u}, \bar{v})$ to {\rm MDP}, we have \eqref{MDP_2} and, by the arbitrariness of $(x, y)$, $(\bar{x}, \bar{y})$ is a globally optimal solution to \eqref{MDVP}. This completes the proof.
\end{proof}

\begin{rem}\rm
Recall that the equivalence conditions between BP and WDP given in \cite{Li2022novel} include the pseudoconvexity of the Lagrange function $L(x,\cdot, u, v)$ and the Guignard constraint qualification. It is not difficult to show by examples that the conditions in the above theorem and the conditions in \cite{Li2022novel} cannot be mutually implied. In other words, the new MDP reformulation may be applicable to different types of bilevel programs from the WDP reformulation.
\end{rem}

\begin{rem}\rm
Even if the lower-level problem does not satisfy the convexity conditions in Theorem \ref{MDP-th}, {\rm BP} is still possibly equivalent to {\rm MDP}, which can be illustrated by the following example.
\end{rem}

\begin{ex}\rm \label{ex-nocondition-MDP}
Consider the bilevel program
\begin{eqnarray}\label{ex_nocondition_MDP_1}		
\min&&2x-y\\
\mbox{\rm s.t.}&&x\geq0,~y\in \arg\min\limits_{y}\{y^3: y\geq x\}.\nonumber
\end{eqnarray}
For any $x\geq0$, since the lower-level objective function is strictly increasing with respect to $y$ on $[x,+\infty)$, we have $S(x)=\{x\}$. Thus, \eqref{ex_nocondition_MDP_1} has a unique optimal solution $(0,0)$ with the optimal value 0. The MDP reformulation of \eqref{ex_nocondition_MDP_1} is
\begin{eqnarray}\label{ex_nocondition_MDP_2}
\min&&2x-y \nonumber \\
\mbox{\rm s.t.} &&x\geq0,~~y\geq x,~~y^3\leq z^3,\\
&&u(z-x)\leq 0,~~3z^2-u=0,~~u\geq 0 \nonumber.
\end{eqnarray}
We next show that, for each $x\geq0$, $f(x,\cdot)$ is not pseudoconvex over $Y(x)\cup\mathcal{Z}(x,u)\equiv\mathbb{R}$ for any $u\geq0$, but the original program \eqref{ex_nocondition_MDP_1} is still equivalent to \eqref{ex_nocondition_MDP_2} in the sense of Theorem \ref{MDP-th}.
		
(i) Taking $z_1=-1$ and $z_2=0$, we have $\nabla_zf(x,z_2)=3{z_2}^2=0$ and $f(x,z_1)=-1 < 0=f(x,z_2).$ This means that
\begin{eqnarray*}
(z_1-z_2)^T\nabla_zf(x,z_2)\geq0 ~\Longrightarrow~ f(x,z_1)\geq f(x,z_2)
\end{eqnarray*}
is not always satisfied. Therefore, $f(x,\cdot)$ is not pseudoconvex over $\mathbb{R}$.
		
(ii) By eliminating the variable $u$, \eqref{ex_nocondition_MDP_2} can be rewritten as
\begin{eqnarray}\label{ex-nocondition-MDP-3}	
\min &&2x-y \\
\mbox{\rm s.t.}&&x\geq0,~~y\geq x,~~y^3\leq z^3,~~3z^2(z-x)\leq 0.\nonumber
\end{eqnarray}	
By $3z^2(z-x)\leq 0$, there are two cases: $z=0$ or $z\leq x$. In the case of $z=0$, the constraints $x\geq0,~y\geq x,~y^3\leq 0$ imply $x=y=0$ and hence $(0,0,0,0)$ is a unique optimal solution to \eqref{ex-nocondition-MDP-3}. In the case of  $z\leq x$, the constraints  $x\geq0,~y\geq x,~y^3\leq z^3,~z\leq x$ imply $x=y=z$ and hence $(0,0,0,0)$ is also a unique optimal solution to \eqref{ex-nocondition-MDP-3}. As a result, $(0,0,0,0)$ is a unique optimal solution to \eqref{ex-nocondition-MDP-3} with the optimal value $0$. Consequently, the bilevel program \eqref{ex_nocondition_MDP_1} is equivalent to its MDP reformulation \eqref{ex_nocondition_MDP_2} in the sense of Theorem \ref{MDP-th}, although the assumptions of Theorem \ref{MDP-th} are not satisfied.
\end{ex}

\begin{thm}\label{MDP-bioptima}
Let $(x, y, z, u, v)$ be feasible to {\rm MDP}. Then, under any conditions to ensure the weak Mond-Weir duality for $\mathrm{P}_x$ and $\mathrm{MD}_x$, $y$ is globally optimal to $\mathrm{P}_x$ and $(z,u,v)$ is globally optimal to $\mathrm{MD}_x$.
\end{thm}
	
\begin{proof}
Since $(x, y, z, u, v)$ is feasible to {\rm MDP}, we have				
\begin{eqnarray}\label{MDP_bioptima_1}	\min\limits_{y'\in Y(x)}f(x,y')\leq f(x, y)\leq f(x, z)\leq \max\limits_{(z',u',v')\in M(x)} f(x,z').
\end{eqnarray}
The weak Mond-Weir duality assumption means $\min\limits_{y'\in Y(x)}f(x,y')\ge \max\limits_{(z',u',v')\in M(x)} f(x,z')$. Therefore, all inequalities in \eqref{MDP_bioptima_1} should be equalities, which shows that $y$ is globally optimal to $\mathrm{P}_x$ and $(z, u, v)$ is globally optimal to $\mathrm{MD}_x$. This completes the proof.
\end{proof}

Theorem \ref{MDP-th} shows the equivalence between BP and MDP in  globally optimal sense. Since they are nonconvex in general, we next consider their equivalence in locally optimal sense. Denote by $U(w)$ some neighborhood of a point $w$.

\begin{thm}\label{MDP-localsolution1}
Suppose that $(\bar{x},\bar{y})$ is locally optimal to {\rm BP} restricted on $U(\bar{x})\times U(\bar{y})$ and $\mathrm{P}_{\bar{x}}$ satisfies the Guignard constraint qualification at some optimal solution $\bar{z}$ with $\{\bar{u},\bar{v}\}$ being the corresponding Lagrange multipliers. Under any conditions to ensure the weak Mond-Weir duality for $\mathrm{P}_x$ and $\mathrm{MD}_x$ and $x\in U(\bar{x})$, $(\bar{x},\bar{y},\bar{z},\bar{u},\bar{v}) $ is locally optimal to {\rm MDP}. Conversely, if $(\bar{x},\bar{y},\bar{z},\bar{u},\bar{v}) $ is locally optimal to {\rm MDP} restricted on $U(\bar{x})\times U(\bar{y}) \times U(\bar{z})\times\mathbb{R}^{p}\times\mathbb{R}^{q}$ with $U(\bar{y}) \subseteq U(\bar{z})$ and $\mathrm{P}_x$ satisfies the Guignard constraint qualification at all points in $S(x)$ for any $x\in U(\bar{x})$, then $(\bar{x},\bar{y}) $ is locally optimal to {\rm BP}.
\end{thm}
	
\begin{proof}		
The proof of this theorem can be divided into two cases: (i) necessity and (ii) sufficiency.

{\rm(i)} Since both $\bar{y}$ and $\bar{z}$ are globally optimal solutions to $\mathrm{P}_{\bar{x}}$ and $\{\bar{u},\bar{v}\}$ are the Lagrange multipliers corresponding to $\bar{z}$, we have
\begin{eqnarray*}
f(\bar{x},\bar{y}) = f(\bar{x},\bar{z}), ~\nabla_z L(\bar{x},\bar{z},\bar{u},\bar{v})=0,~\bar{u}^T g(\bar{x},\bar{z})=0,~ \bar{u}\geq0,~h(\bar{x},\bar{z})=0	
\end{eqnarray*}
and hence $(\bar{x},\bar{y},\bar{z},\bar{u},\bar{v}) $ is feasible to MDP. On the other hand, for any point $(x,y,z,u,v)$ feasible to MDP in the neighborhood $U(\bar{x})\times U(\bar{y})\times U(\bar{z})\times U(\bar{u})\times U(\bar{v})$ with $U(\bar{z})\times U(\bar{u})\times U(\bar{v})$ to be a neighbourhood of $(\bar{z},\bar{u},\bar{v})$, by the weak Mond-Weir duality assumption, the conclusion of Theorem \ref{MDP-bioptima} is true and hence $y\in S(x)$, which means that $(x,y)$ is feasible to BP and belongs to $U(\bar{x})\times U(\bar{y})$. Noting that $(\bar{x},\bar{y})$ is locally optimal to BP restricted on $U(\bar{x})\times U(\bar{y})$, we have
\begin{eqnarray}\label{MDP_localsolution_1}
F(\bar{x},\bar{y}) \leq F(x, y).
\end{eqnarray}
By the arbitrariness of $(x,y,z,u,v)$, $(\bar{x},\bar{y},\bar{z},\bar{u},\bar{v}) $ is a locally optimal solution to MDP.
	
{\rm(ii)} Let $(x,y)\in U(\bar{x})\times U(\bar{y})$ be feasible to BP. It follows that $y\in S(x)$. By the assumptions, $\mathrm{P}_x$ satisfies the Guignard constraint qualification at $y$ and hence there exist multipliers $\{u,v\}$ satisfying
\begin{eqnarray*}
\nabla_y L(x, y,u,v) = 0, ~u\geq 0, ~u^{T} g(x, y)=0,~h(x,y)=0.
\end{eqnarray*}
Consequently, $(x,y,y,u,v)$ is feasible to MDP and belongs to $U(\bar{x})\times U(\bar{y})\times U(\bar{y})\times U(\bar{u})\times U(\bar{v})$. Since $(\bar{x},\bar{y},\bar{z},\bar{u},\bar{v}) $ is locally optimal to MDP restricted on $U(\bar{x})\times U(\bar{y}) \times U(\bar{z})\times\mathbb{R}^{p}\times\mathbb{R}^{q}$ with $U(\bar{y}) \subseteq U(\bar{z})$, we have \eqref{MDP_localsolution_1}. By the arbitrariness of $(x,y)$, $(\bar{x},\bar{y}) $ is a locally optimal solution to BP. This completes the proof.
\end{proof}

The local optimality restricted on $U(\bar{x})\times U(\bar{y}) \times U(\bar{z})\times \mathbb{R}^{p}\times\mathbb{R}^{q}$ assumed in Theorem \ref{MDP-localsolution1} is very stringent. Next, we derive a result by restricting the local optimality on $U(\bar{x})\times U(\bar{y}) \times U(\bar{y})\times U(\bar{u})\times U(\bar{v})$. We start with the following lemma given in \cite{Li2022novel}.

\begin{lema}\label{lemma-usc}
The set-valued mapping $\Lambda : \mathbb{R}^{n+2m}\rightrightarrows \mathbb{R}^{p+q}$ defined by
\begin{eqnarray*}\label{Lambda}
\Lambda(x, y, z) := \big\{(u, v)\in\mathbb{R}^{p+q} : (x, y, z, u, v)\in \Gamma \big\}
\end{eqnarray*}
is outer semicontinuous, where		
\begin{eqnarray*}
			\Gamma:=\left\{ (x, y, z, u, v) :
			\begin{array}{l}
				f(x, y) - f(x, z) \leq 0,~\nabla _z L(x, z, u, v) = 0\\
				u^T g(x, z)+ v^T h(x,z)\geq0, ~u\ge 0
			\end{array}
			\right\}.
      \end{eqnarray*}
\end{lema}

\begin{thm}\label{MDP-localsolution2}
Let $(\bar{x},\bar{y},\bar{y},\bar{u},\bar{v}) $ be locally optimal to {\rm MDP} for any $(\bar{u},\bar{v})\in \Lambda(\bar{x},\bar{y},\bar{y})$. Under any conditions to ensure the weak Mond-Weir duality for $\mathrm{P}_{\bar{x}}$ and $\mathrm{MD}_{\bar{x}}$, if $\mathrm{P}_x$ satisfies the MFCQ at ${y}\in S({x})$ for all $(x,y)$ near $(\bar{x},\bar{y})$, $(\bar{x},\bar{y})$ is locally optimal to {\rm BP}.
\end{thm}

\begin{proof}
Since $(\bar{x},\bar{y},\bar{y},\bar{u},\bar{v}) $ is feasible to MDP, it follows from the weak Mond-Weir duality and Theorem \ref{MDP-bioptima} that
$\bar{y}\in S(\bar{x})$ and hence $(\bar{x},\bar{y})$ is  feasible to BP. We next show the conclusion by contradiction.
	
Suppose that $(\bar{x},\bar{y})$ is not locally optimal to BP. There must exist a sequence $\{(x^k, y^k)\}$ converging to $(\bar{x},\bar{y})$ with $x^k\in X$, $y^k\in S(x^k)$ such that
\begin{eqnarray}\label{local_inequality}
F(x^k, y^k) < F(\bar{x},\bar{y}), \quad k=1,2,\cdots.
\end{eqnarray}
By the assumptions, when $k$ is large sufficiently, $\mathrm{P}_x$ satisfies the MFCQ at ${y^k}\in S({x^k})$, namely, there exist multipliers $u^k\in\mathbb{R}^p$ and $v^k\in\mathbb{R}^q$ such that
\begin{eqnarray}
\label{stablesystem1}&&\nabla _y L(x^k, y^k, u^k, v^k) = 0, ~h(x^k, y^k)=0,\\
\label{stablesystem2}&&u^k\ge 0,~g(x^k, y^k) \leq 0,~(u^k)^Tg(x^k, y^k)=0 .	
\end{eqnarray}
We can show that the sequence of $\{(u^k,v^k)\}$ is bounded. Suppose by contradiction that $\{(u^k,v^k)\}$ is unbounded. Without loss of generality, we assume that $\{\frac{u^k}{\|u^k\|+\|v^k\|}\}$ and $\{\frac{v^k}{\|u^k\|+\|v^k\|}\}$ are  convergent. Denote their limits by $\hat{u}$ and $\hat{v}$, respectively. It follows that
\begin{eqnarray}\label{uv_inequality}
\|\hat{u}\|+\|\hat{v}\|=1.
\end{eqnarray}
Dividing by $\|{u}^k\|+\|{v}^k\|$ on both sides of \eqref{stablesystem1}-\eqref{stablesystem2} and letting $k\rightarrow\infty$ yield
\begin{eqnarray}
\label{stablesystem3_1}
&&\nabla_yg(\bar{x},\bar{y})\hat{u}+\nabla_y h(\bar{x},\bar{y})\hat{v}=0,\\
\label{stablesystem3_2}&&
\hat{u}\geq 0, ~g(\bar{x},\bar{y})\le 0, ~\hat{u}^{T} g(\bar{x},\bar{y}) = 0.
\end{eqnarray}
Let ${\cal I}_g$ be the active index set of $g$ at $(\bar{x},\bar{y})$. By \eqref{stablesystem3_2}, \eqref{stablesystem3_1} reduces to
\begin{eqnarray}\label{stablesystem3_3}		
\sum_{i\in {\cal I}_g} \hat{u}_i\nabla_y g_i(\bar{x},\bar{y})+\sum_{j=1}^q \hat{v}_j\nabla_y h_j(\bar{x},\bar{y})=0.
\end{eqnarray}
Since $\mathrm{P}_{\bar{x}}$ satisfies the MFCQ at $\bar{y}$, the gradients $\nabla_y h_j(\bar{x},\bar{y})~(j=1, \cdots, q) $ are linearly independent and there exists $ d\in\mathbb{R}^m $ such that
$d^T\nabla _y h_j(\bar{x},\bar{y}) =0$ for each $j$ and $d^T \nabla_y g_i(\bar{x}, \bar{y})<0$ for each $i\in{\cal I}_g$. Taking the inner product on both sides of \eqref{stablesystem3_3} with $d$, we have
\begin{eqnarray*}
0=\sum_{i\in {\cal I}_g} \hat{u}_i d^T\nabla_y g_i(\bar{x},\bar{y})<0,
\end{eqnarray*}
which implies ${\cal I}_g=\emptyset$ and hence $\hat{u}=0$ by \eqref{stablesystem3_2}. Thus, \eqref{stablesystem3_1} becomes $\nabla_y h(\bar{x},\bar{y})\hat{v}=0$, which implies $\hat{v}=0$ because the gradients $ \nabla_y  h_j(\bar{x},\bar{y})~(j=1, \cdots, q) $ are linearly independent. This contradicts \eqref{uv_inequality}. Therefore, the sequence of $\{(u^k,v^k)\}$ is bounded.
	
Since $(x^k, y^k, y^k) \to (\bar{x},\bar{y},\bar{y})$ as $k \to \infty$, by the outer semicontinuousness of $\Lambda$ from Lemma \ref{lemma-usc}, the sequence $\{(u^k, v^k)\}$ has at least an accumulation point $(\tilde{u},\tilde{v})\in \Lambda(\bar{x},\bar{y},\bar{y})$ as $k\to\infty$. By the assumptions, the limit point $(\bar{x}, \bar{y}, \bar{y}, \tilde{u}, \tilde{v})$ is also locally optimal to MDP. Note that, from the definition of $\Lambda$, $(x^k, y^k, y^k, u^k, v^k)$ is feasible to MDP for each $k$. Thus, the inequality \eqref{local_inequality} is a contradiction. Consequently, $(\bar{x},\bar{y})$ must be a locally optimal solution to BP. This completes the proof.
\end{proof}

\section{MDP may satisfy the MFCQ at its feasible points}\label{conterex}
As is known to us, MPCC fails to satisfy the MFCQ at any feasible point so that the well-developed optimization algorithms in nonlinear programming may not be stable in solving MPCC. Compared with MPCC, it has been shown in \cite{Li2022novel} that WDP may satisfy the MFCQ at its feasible points. The following example indicates that MDP also has this advantage over MPCC.
	
\begin{ex}\rm\label{ex-MDP-MFCQ}
Consider the bilevel program
\begin{eqnarray}\label{ex_MDP_MFCQ_1}	      \min&&(x+y)^2\\
\mbox{s.t.}&&x\in[-1,1] ,~y \in \arg\min\limits_{y}\{y^3-3y:y\geq x\}.\nonumber
\end{eqnarray}
It is evident that $S(x)=\{1\}$ for any $x\in [-1,1]$ and so \eqref{ex_MDP_MFCQ_1} has a unique optimal solution $(-1,1)$. The MDP reformulation for \eqref{ex_MDP_MFCQ_1} is
\begin{eqnarray}\label{ex_MDP_MFCQ_2}
\min &&(x+y)^2 \nonumber\\
\mbox{s.t.}&&y^3-3y-z^3+3z\leq 0,~u(z-x)\leq0,~u\geq0 ,\\
&&3z^2-3-u=0,~y\geq x,~x\in[-1,1].\nonumber
\end{eqnarray}
It is not difficult to verify that $(x^*, y^*, z^*, u^*)=(-1, 1, -2,9)$ is a globally optimal solution to \eqref{ex_MDP_MFCQ_2}. In what follows, we show that the MFCQ holds at this point. In fact, the gradient corresponding to the unique equality constraint $3z^2-3-u=0$ at this point is
\begin{eqnarray*}
a_1:=(0,0,6z^*,-1)=(0,0,-12,-1),
\end{eqnarray*}
which is a nonzero vector, and hence it is linearly independent. Note that the active inequality constraints at this point contain $y^3-3y-z^3+3z\leq 0,~x\geq -1$ and the corresponding gradients at the point are respectively
\begin{eqnarray*}
b_1:=(0,3{y^*}^2-3,-3{z^*}^2+3,0)=(0,0,-9,0), \quad b_2:=(-1,0,0,0).
\end{eqnarray*}
Let $d_1:=(1,0,1,-12)$. Then, we have
\begin{eqnarray*}
d_1^Ta_1=0,\quad d_1^Tb_1<0,\quad d_1^Tb_2<0.
\end{eqnarray*}
This means that \eqref{ex_MDP_MFCQ_2} satisfies the MFCQ at the feasible point $(-1, 1, -2,9)$.
\end{ex}

This example reveals that, in theory at least, MDP is relatively easier to solve than MPCC. However, under some conditions, the MFCQ may also fail to hold for MDP, as shown in the following result.

\begin{thm}\label{MDP-MFCQ-Fail}
Let $(x, y, z, u, v)$ be a feasible point of {\rm MDP}. Then, the MFCQ fails to hold at this point if $z=y$.
\end{thm}
			
\begin{proof}
For simplicity, we take the upper constraint $(x,y)\in \Omega$ away from $\mathrm{MDP}$. The set of abnormal multipliers for {\rm MDP} at $(x, y, z, u, v)$ is given by
\begin{eqnarray*}
\mathcal{M}:=\left\{\begin{array}{lll}
			&{0= \alpha(\nabla_x f(x, y)-\nabla_x f(x, z))+\nabla^2_{zx}L(x,z,u,v)\beta}\\[1mm]
			&{~~~~~ +\nabla_x g(x, y) (\eta^g-\gamma u)+\nabla_x h(x, y) (\eta^h-\gamma v) }\\[1mm]
			&{0=\alpha\nabla_y f(x, y)+ \nabla_y g(x, y)\eta^g+ \nabla_y h(x, y)\eta^h}\\[1mm]
			{(\alpha, \beta, \gamma, \eta^g, \eta^h, \eta^u):\ }&{0=-\alpha\nabla_z f(x,z)+\nabla^2_{zz}L(x,z,u,v)\beta-\gamma\nabla_z g(x, z)u-\gamma\nabla_z h(x, z)v}\\[1mm]
			&{0=\nabla_z g(x, z)^{T}\beta-\gamma g(x, z)-\eta^u}\\[1mm]
			&{0=\nabla_z h(x, z)^{T}\beta-\gamma h(x, z)}\\[1mm]
			&{\alpha\geq 0,~\alpha(f(x, y)-f(x, z))=0,~~\eta^u\geq 0,~u^T \eta^u =0} \\[1mm]
			&{\gamma\geq 0,~\gamma (u^T g(x,z)+v^T h(x,z))=0,~~\eta^g\geq 0,~ g(x, y)^T\eta^g =0}
		\end{array}\right\}.	
\end{eqnarray*}
If $z=y$, it is not difficult to verify by the feasibility of $(x, y, y, u)$ to {\rm MDP} that $( 1, 0, 1, u, v, -g(x, y))\in \mathcal{M}$.
Recall that, by \cite[Chapter 6.3]{Clarke1990optimization}, a nonzero abnormal multiplier exists if and only if the MFCQ does not hold. Therefore, the MFCQ does not hold at $(x, y, y, u,v)$ for {\rm MDP}. This completes the proof.
\end{proof}

\section{Relations among MPCC, WDP, and MDP}
\label{Comparison of three reformulations}
So far, we have three different single-level reformulations for BP, namely, MPCC, WDP, and MDP. It is meaningful and important to explore the relations among them. For simplicity, throughout this section, we take the upper constraint $(x,y)\in \Omega$ away from the original program BP. Actually, there is no any difficulty to extend the subsequent analysis to the general case.

We mainly investigate the relations between stationary points of MPCC, WDP, and MDP, since they are all nonconvex optimization problems in general.
	
\begin{thm}\label{relations-KKT}
Let $(\bar{x},\bar{y},\bar{z},\bar{u},\bar{v})$ be a KKT point of {\rm WDP}. If $(\bar{x}, \bar{y}, \bar{z}, \bar{u},\bar{v})$ is feasible to {\rm MDP}, it is also a KKT point of {\rm MDP}.
\end{thm}
		
\begin{proof}
Since $(\bar{x},\bar{y},\bar{z},\bar{u},\bar{v})$ is a KKT point of WDP,  there is $(\zeta^g, \zeta^h, \zeta^u, \hat{\alpha}, \hat{\beta}) \in \mathbb{R}^{2p+q+1+m}$ such that
\begin{eqnarray}	
&& \nabla_x  F(\bar{x},\bar{y}) +\hat{\alpha}(\nabla_x  f(\bar{x},\bar{y})-\nabla_x L(\bar{x},\bar{z}, \bar{u},\bar{v}))+\nabla_{zx}^2 L(\bar{x},\bar{z}, \bar{u},\bar{v})\hat{\beta}
    		\nonumber\\
    		\label{WDPKKT-1}&&+\nabla_x g(\bar{x},\bar{y}) \zeta^g+\nabla_x h(\bar{x},\bar{y}) \zeta^h =0,  \\	    	
    		\label{WDPKKT-2} &&\nabla_y F(\bar{x},\bar{y}) + \hat{\alpha}\nabla_y f(\bar{x},\bar{y}) + \nabla_y g(\bar{x},\bar{y})\zeta^g
    		+\nabla_y h(\bar{x},\bar{y})\zeta^h=0, \\     	\label{WDPKKT-3}&&\nabla_{zz}^2 L(\bar{x},\bar{z}, \bar{u},\bar{v}) \hat{\beta}=0, \\
    		\label{WDPKKT-4}&&-\hat{\alpha} g(\bar{x},\bar{z}) + \nabla_z g(\bar{x},\bar{z})^T\hat{\beta} - \zeta^u=0, \\
    		\label{WDPKKT-5}&&-\hat{\alpha} h(\bar{x},\bar{z})+\nabla_z h(\bar{x},\bar{z})^T\hat{\beta} =0,\\
    		\label{WDPKKT-6}&&0\le \widehat{\alpha} ~\bot ~f(\bar{x},\bar{y})-L(\bar{x},\bar{z},\bar{u},\bar{v}) \leq0,  \\
    		\label{WDPKKT-7}&&0\le \zeta^g ~\bot ~g(\bar{x},\bar{y})\leq 0,  \\
    		\label{WDPKKT-8}&&0\le \zeta^u ~\bot ~\bar{u} \geq 0,\\
    		\label{WDPKKT-9}&&\nabla_z L(\bar{x},\bar{z}, \bar{u},\bar{v})=0,~h(\bar{x},\bar{y})=0.
    	\end{eqnarray}
Setting
$\eta^g:=\zeta^g,~\eta^h:=\zeta^h,~\eta^u:=\zeta^u,~\alpha:=\widehat{\alpha},~\beta:=\hat{\beta},$ and $\gamma:=\widehat{\alpha}$, we have
\begin{eqnarray}		
&&\nabla_x  F(\bar{x},\bar{y})  + \alpha(\nabla_x f(\bar{x},\bar{y})-\nabla_x f(\bar{x},\bar{z}))-\gamma(\nabla_x g(\bar{x},\bar{z})\bar{u}+\nabla_x h(\bar{x},\bar{z})\bar{v}) \nonumber \\
		\label{MDPKKT-1}&&+\nabla_{zx}^2 L(\bar{x},\bar{z}, \bar{u},\bar{v}) \beta  +\nabla_x g(\bar{x},\bar{y}) \eta^g+\nabla_x h(\bar{x},\bar{y}) \eta^h=0,  \\
		\label{MDPKKT-2}&&\nabla_y F(\bar{x},\bar{y}) + \alpha\nabla_y f(\bar{x},\bar{y}) + \nabla_y g(\bar{x},\bar{y})\eta^g+ \nabla_y h(\bar{x},\bar{y})\eta^h=0, \\
		\label{MDPKKT-3}&&-\alpha\nabla_z f(\bar{x},\bar{z})-\gamma(\nabla_z g(\bar{x},\bar{z})\bar{u}+\nabla_z h(\bar{x},\bar{z})\bar{v})+\nabla_{zz}^2 L(\bar{x},\bar{z}, \bar{u},\bar{v}) \beta=0, \\
		\label{MDPKKT-4}&&-\gamma g(\bar{x},\bar{z})+\nabla_y g(\bar{x},\bar{z})^T\beta - \eta^u=0, \\
		\label{MDPKKT-5}&&-\gamma h(\bar{x},\bar{z})+\nabla_z h(\bar{x},\bar{z})^T\beta=0, \\
		\label{MDPKKT-8}&&0\le \eta^g ~\bot ~g(\bar{x},\bar{y})\leq 0,  \\
		\label{MDPKKT-9}&&0\le \eta^u ~\bot ~\bar{u} \geq 0,\\
		\label{MDPKKT-10}&&\nabla_z L(\bar{x},\bar{z}, \bar{u},\bar{v})=0,~h(\bar{x},\bar{y})=0.
\end{eqnarray}
By \eqref{WDPKKT-6}, we further have
\begin{eqnarray*}	\widehat{\alpha}(f(\bar{x},\bar{y})-L(\bar{x},\bar{z},\bar{u},\bar{v}))=		\alpha(f(\bar{x},\bar{y})-f(\bar{x},\bar{z}))+\gamma (-\bar{u}^Tg(\bar{x},\bar{z})-\bar{v}^Th(\bar{x},\bar{z}))=0.
\end{eqnarray*}
Since $(\bar{x},\bar{y},\bar{z},\bar{u},\bar{v})$ is feasible to MDP, we have $f(\bar{x},\bar{y})-f(\bar{x},\bar{z})\leq0$ and $\bar{u}^Tg(\bar{x},\bar{z})+\bar{v}^Th(\bar{x},\bar{z})\geq0$, which implies
\begin{eqnarray}
\label{MDPKKT-6}&&0\leq\alpha ~\bot ~f(\bar{x},\bar{y})- f(\bar{x},\bar{z})\leq 0, \\
\label{MDPKKT-7}&&0\le \gamma ~\bot ~\bar{u}^{T} g(\bar{x},\bar{z})+\bar{v}^{T} h(\bar{x},\bar{z}) \geq0
\end{eqnarray}
by the nonnegativity of $\alpha$ and $\gamma$. \eqref{MDPKKT-1}-\eqref{MDPKKT-7} means that $(\bar{x}, \bar{y}, \bar{z}, \bar{u},\bar{v})$ is a KKT point of {\rm MDP}. This completes the proof.
\end{proof}

Given a feasible point $(\bar{x},\bar{y},\bar{u},\bar{v})$ of MPCC, we define the following index sets:
\begin{eqnarray*}
\ \ I_{0+}:=\{i : g_{i}(\bar{x},\bar{y})=0, ~\bar{u}_i>0\},\ \
I_{-0}:=\{i : g_{i}(\bar{x},\bar{y})<0, ~\bar{u}_i=0\},\ \
I_{00}:=\{i : g_{i}(\bar{x},\bar{y})=0=\bar{u}_i\}.
\end{eqnarray*}

\begin{defi}\rm
We call $ (\bar{x},\bar{y},\bar{u},\bar{v})$ a strongly stationary point (or S-stationary point) of MPCC if there exists $(\lambda^g,\lambda^h,\lambda^u,\lambda^L)\in\mathbb{R}^{2p+q+m}$ such that
\begin{eqnarray}
\label{MPCC-S-1}&& \nabla_x F(\bar{x},\bar{y}) + \nabla_{yx}^2 L(\bar{x},\bar{y}, \bar{u},\bar{v}) \lambda^L  + \nabla_x g(\bar{x},\bar{y}) \lambda^g+ \nabla_x h(\bar{x},\bar{y}) \lambda^h=0,\\
			\label{MPCC-S-2}&&\nabla_y F(\bar{x},\bar{y}) + \nabla_{yy}^2 L(\bar{x},\bar{y}, \bar{u},\bar{v})\lambda^L  + \nabla_y g(\bar{x},\bar{y}) \lambda^g+ \nabla_y h(\bar{x},\bar{y}) \lambda^h=0, \\
			\label{MPCC-S-3}&&\nabla_y g(\bar{x},\bar{y})^T \lambda^L  -\lambda^u= 0, \\
			\label{MPCC-S-4}&&\nabla_y h(\bar{x},\bar{y})^T \lambda^L = 0, \\
			\label{MPCC-S-5}&&\nabla_{y} L(\bar{x},\bar{y}, \bar{u},\bar{v}) =0, ~g(\bar{x},\bar{y})\le 0, ~\bar{u}\ge 0,~h(\bar{x},\bar{y})=0,\\
			\label{MPCC-S-6}&&\lambda^g_i = 0, ~~i \in I_{-0}, \\
			\label{MPCC-S-7}&&\lambda^u_i = 0, ~~i \in I_{0+},\\
			\label{MPCC-S-8}&&\lambda^g_i \geq 0, ~\lambda^u_i \geq 0, ~~i\in I_{00}.
\end{eqnarray}
\end{defi}

\begin{rem}\rm
Other popular stationarity concepts for {\rm MPCC} include the Mordukhovich stationarity (or M-stationarity) and the Clarke stationarity (or C-stationarity). Among these stationarities, the S-stationarity is the strongest and the most favorable one in the study of {\rm MPCC}.
\end{rem}

\begin{thm}\label{KKT=S}
If $(\bar{x},\bar{y},\bar{y},\bar{u},\bar{v})$ is a KKT point of {\rm MDP}, then $(\bar{x}, \bar{y}, \bar{u},\bar{v})$ is an S-stationary point of {\rm MPCC}.
\end{thm}

\begin{proof}
Since $(\bar{x},\bar{y},\bar{y},\bar{u},\bar{v})$ is a KKT point of {\rm MDP},  there is $(\eta^g, \eta^h, \eta^u, \alpha, \beta,\gamma) \in \mathbb{R}^{2p+q+2+m}$ such that
	\begin{eqnarray}
		\label{MDPKKT2-1}&&\nabla_x  F(\bar{x},\bar{y})  + \nabla_{yx}^2 L(\bar{x},\bar{y}, \bar{u},\bar{v}) \beta  +\nabla_x g(\bar{x},\bar{y}) (\eta^g-\gamma\bar{u})+\nabla_x h(\bar{x},\bar{y}) (\eta^h-\gamma\bar{v})=0,  \\
		\label{MDPKKT2-2}&&\nabla_y F(\bar{x},\bar{y}) + \alpha\nabla_y f(\bar{x},\bar{y}) + \nabla_y g(\bar{x},\bar{y})\eta^g+ \nabla_y h(\bar{x},\bar{y})\eta^h=0, \\
		\label{MDPKKT2-3}&&-\alpha\nabla_y f(\bar{x},\bar{y})+\nabla_{yy}^2 L(\bar{x},\bar{y}, \bar{u},\bar{v}) \beta- \gamma\nabla_y g(\bar{x},\bar{y}) \bar{u}- \gamma\nabla_y h(\bar{x},\bar{y}) \bar{v}=0, \\
		\label{MDPKKT2-4}&&\nabla_y g(\bar{x},\bar{y})^T\beta -\gamma g(\bar{x},\bar{y})- \eta^u=0, \\
		\label{MDPKKT2-5}&&\nabla_y h(\bar{x},\bar{y})^T\beta=0, \\
		\label{MDPKKT2-6}&&\nabla_y L(\bar{x},\bar{y}, \bar{u},\bar{v})=0,~h(\bar{x},\bar{y})=0, ~\alpha\geq0, ~0\le \gamma ~\bot ~\bar{u}^{T} g(\bar{x},\bar{y})+\bar{v}^{T} h(\bar{x},\bar{y}) \geq0, \\
		\label{MDPKKT2-7}&&0\le \eta^g ~\bot ~g(\bar{x},\bar{y})\leq 0,  \\
		\label{MDPKKT2-8}&&0\le \eta^u ~\bot ~\bar{u} \geq 0.
\end{eqnarray}
Adding \eqref{MDPKKT2-2} and \eqref{MDPKKT2-3} together yields
\begin{eqnarray}\label{MDPKKT2-9}		
\nabla_y F(\bar{x},\bar{y})+\nabla_{yy}^2 L(\bar{x},\bar{y}, \bar{u},\bar{v}) \beta + \nabla_y g(\bar{x},\bar{y})(\eta^g-\gamma\bar{u})+ \nabla_y h(\bar{x},\bar{y})(\eta^h-\gamma\bar{v})=0.
\end{eqnarray}
Letting $\lambda^g:=\eta^g-\gamma\bar{u},~\lambda^h:=\eta^h-\gamma\bar{v}, ~\lambda^u:=\eta^u+\gamma g(\bar{x},\bar{y})$, and $\lambda^L:=\beta$, we obtain \eqref{MPCC-S-1}-\eqref{MPCC-S-5} from \eqref{MDPKKT2-1} and \eqref{MDPKKT2-4}-\eqref{MDPKKT2-9} immediately. It remains to show \eqref{MPCC-S-6}-\eqref{MPCC-S-8}. In fact, if $i\in I_{-0}$, which means $g_{i}(\bar{x},\bar{y})<0$ and $\bar{u}_i=0$, we have $\eta_i^g=0$ by \eqref{MDPKKT2-7} and hence $\lambda_i^g=\eta_i^g-\gamma \bar{u}_i=0$, namely, \eqref{MPCC-S-6} holds. If $i\in I_{0+}$, which means $g_{i}(\bar{x},\bar{y})=0$ and $\bar{u}_i>0$, we have $\eta_i^u=0$ by \eqref{MDPKKT2-8} and hence $\lambda_i^u=\eta_i^u+\gamma g_{i}(\bar{x},\bar{y})=0$, namely, \eqref{MPCC-S-7} holds. If $i\in I_{00}$, which means $g_{i}(\bar{x},\bar{y})=\bar{u}_i=0$, we have $\lambda_i^g=\eta_i^g\geq 0$ and $\lambda_i^u=\eta_i^u \geq 0$ by \eqref{MDPKKT2-7} and \eqref{MDPKKT2-8}, respectively, namely, \eqref{MPCC-S-8} holds.
This completes the proof.
\end{proof}

Note that the converse of Theorem \ref{KKT=S} is incorrect, which can be illustrated by the following example.

\begin{ex}\rm\label{ex-MDP-KKT<S}
Consider the bilevel program %
\begin{eqnarray}\label{ex-MDP-KKT<S-1}						\min&& x^2-(2y+1)^2\\
\mbox{s.t.}&& x\leq0,~y\in \arg\min\limits_{y}\{(y-1)^2: 3x-y-3\leq0,~x+y-1\leq0\}.\nonumber
\end{eqnarray}
It is shown in \cite{Li2022novel} that $(0,1)$ is a unique optimal solution of \eqref{ex-MDP-KKT<S-1} and $(0,1,0,0)$ is an S-stationary point of the corresponding MPCC.
Since the lower-level problem is a convex quadratic programming, we can equivalently transform \eqref{ex-MDP-KKT<S-1} into the corresponding MDP
\begin{eqnarray}\label{ex-MDP-KKT<S-2}				
\min&& x^2-(2y+1)^2 \nonumber\\
\mbox{s.t.} &&x\leq0, ~3x-y-3\leq0, ~x+y-1\leq0,\\			&&(y-1)^2-(z-1)^2\leq0,~u_1(3x-z-3)+u_2(x+z-1)\geq 0, \nonumber\\
&&u_1\ge 0,~u_2\geq0,~2(z-1)-u_1+u_2=0 . \nonumber
\end{eqnarray}
Next, we show that $(0,1,1,0,0)$ is not a KKT point of \eqref{ex-MDP-KKT<S-2}. Suppose by contradiction that there exists
$(\eta^G,\eta^g,\eta^u,\alpha,\beta,\gamma)\in\mathbb{R}^8$ such that
\begin{eqnarray*}
&&\eta^G+3\eta_1^g+\eta_2^g=0,~ -12-\eta_1^g+\eta_2^g=0,~2\beta=0,\nonumber\\			&&-\beta+4\gamma-\eta_1^u=0,~\beta-\eta_2^u=0,\\
&&~\eta^G\geq0,~\eta_1^g=0,~\eta_2^g\geq0,~\eta^u\geq0,~\alpha \geq0,~\gamma\geq0. \nonumber
\end{eqnarray*}
We have
\begin{eqnarray*}
\eta_1^g=0~~\Longrightarrow~~ \eta^G+\eta_2^g=0,~\eta_2^g=12~~\Longrightarrow ~~\eta^G=-12,
\end{eqnarray*}
which contradicts $\eta^G \geq0$. Therefore, $(0,1,1,0,0)$ is not a KKT point of the MDP \eqref{ex-MDP-KKT<S-2}.
\end{ex}
		
Notably, it was shown in \cite{Li2022novel} that, if $(\bar{x},\bar{y},\bar{y},\bar{u},\bar{v})$ is a KKT point of {\rm WDP}, it is also an S-stationary point of {\rm MPCC}, but the inverse is not true yet.

\section{Extension of MDP} \label{extension}
In this part, we define a new dual problem of $\mathrm{P}_x$ as
\begin{eqnarray*}\label{eMDx}
{(\mathrm{eMD}_x)}\qquad\max\limits_{z, u,v}&&f(x,z) \nonumber\\
\mbox{s.t.}&& \nabla_z L(x, z, u, v)=0,~u\geq 0, ~ u\circ g(x,z)\ge 0, ~ v\circ h(x,z)= 0,	
\end{eqnarray*}
where $\circ$ denotes the Hadamard product. Denote by $M'(x)$ the feasible region of $\mathrm{eMD}_x$ and, for any fixed $(x,u,v)$, let	
\begin{eqnarray*}
\mathcal{Z}'(x,u,v):=\{z\in \mathbb{R}^m: \nabla_z L(x,z, u,v)=0, ~ u\circ g(x,z)\ge 0,~ v\circ h(x,z)= 0\}.			
\end{eqnarray*}

For $\mathrm{P}_x$ and $\mathrm{eMD}_x$, the weak and strong duality theorems still hold; see Appendix for details. Thus, in a similar way to the proof of Theorem \ref{MDP-th}, we can show the following equivalent result.
\begin{thm}\label{eMDP-th}
Suppose that, for any given $x\in X$, $u\in \mathbb{R}^p_+$, and $v\in \mathbb{R}^q$, $f(x,\cdot)$ is pseudoconvex and $u^{T} g(x, \cdot)+v^{T} h(x, \cdot)$ is quasiconvex over $Y(x)\cup \mathcal{Z}'(x,u,v)$, and $\mathrm{P}_x$ satisfies the Guignard constraint qualification at some solution $y_x \in {S(x)}$. Then, if $(\bar{x},\bar{y})$ is an optimal solution to {\rm BP}, there exists $(\bar{z},\bar{u},\bar{v})$ such that $(\bar{x},\bar{y},\bar{z},\bar{u},\bar{v})$ is an optimal solution to
\begin{eqnarray*}
{\rm (eMDP)}\qquad\min&&F(x, y)  \nonumber\\
\mbox{\rm s.t.}&& (x,y)\in \Omega,~g(x, y) \le 0, ~h(x, y) =0,\\
&& f(x, y) - f(x, z)\leq 0, ~ u\circ g(x,z)\ge 0,~ v\circ h(x,z)= 0,\\
&&\nabla_z f(x, z) + \nabla_z g(x, z) u+ \nabla_z h(x, z) v = 0, ~u\geq 0. \nonumber
\end{eqnarray*}
Conversely, if $(\bar{x},\bar{y},\bar{z},\bar{u},\bar{v})$ is an optimal solution to {\rm eMDP}, then $(\bar{x},\bar{y})$ is an optimal solution to {\rm BP}.
\end{thm}

Note that, like MDP, eMDP may also satisfy the MFCQ at its feasible points, which can be illustrated by Example \ref{ex-MDP-MFCQ}.

\section{Numerical experiments}
\label{numerical}
In this section, we compare the new MDP and eMDP approaches with the MPCC and WDP approaches by generating some tested problems randomly. All experiments were implemented in MATLAB 9.13.0 and run on a computer with Windows 10, Intel(R) Xeon(R) E-2224G CPU $@$ 3.50GHz 3.50 GHz. All implementation codes and detailed numerical results are available in the Github repository \cite{Li2023github}.

\subsection{Test examples}
\label{Test examples}
We designed a procedure to generate the following linear bilevel program:
\begin{eqnarray}\label{linear-linear BP}	
	\min&&c_1^T x+c_2^T y\nonumber\\
	\mbox{s.t.}&&A_1x\leq b_1,\\
	&&y\in \arg\min\limits_{y}\{d_2^T y: \!A_2x+B_2y \leq b_2,b_l\leq y\leq b_u\},\nonumber
\end{eqnarray}
where all coefficients $A_1\in \mathbb{R}^{l\times n}$, $A_2\in \mathbb{R}^{p\times n}$, $B_2\in \mathbb{R}^{p\times m}$, $b_1\in \mathbb{R}^l$, $b_2\in \mathbb{R}^p$, $c_1\in \mathbb{R}^n$, and $c_2,d_2\in \mathbb{R}^m$ were generated by \emph{sprand} in $[-1,1]$ with density to be 0.5, while the lower and upper bounds were generated by $b_l=-10*ones(m,1)$ and $b_u=10*ones(m,1)$, respectively.

In our experiments, we generated 150 problems, which contain three groups by adjusting only the lower-level dimensions $m$ and $p$ because the lower-level variables and constraints are the essential factors in solving bilevel programs: The first 50 problems were generated with $\{n=20, l=30, m=60, p=50\}$, the second 50 problems were generated with $\{n=20, l=30, m=100, p=110\}$, and the last 50 problems were generated with $\{n=20, l=30, m=160, p=150\}$.

When evaluating the numerical results by different algorithms, we mainly focused on three aspects: The top priority was the feasibility of the approximation solutions to the original problem \eqref{linear-linear BP}, the second priority was to compare their upper-level objective values if their feasibilities satisfied the given accuracy requirements, and the third priority was to compare CPU times. Similarly as in \cite{Li2022novel}, we adopted the following unified criterion to measure the infeasibility to \eqref{linear-linear BP} in our tests:
\begin{eqnarray*}	{\rm{Infeasibility}}&\!\!:=\!\!&\|\max(0,A_1x^k-b_1)\|+\|\max(0, A_2x^k+B_2y^k-b_2)\|\\
&&+\|\max(0,y^k-b_u)\|+\|\max(0,b_l-y^k)\|+|d_2^{T}y^k-V(x^k)|,
\end{eqnarray*}
where $(x^k,y^k)$ is the current approximation solution and $V(x^k):=\min \{d_2^{T}y: A_2x^k+B_2y\leq b_2,~b_l\leq y\leq b_u\}$. Obviously, $(x^k,y^k)$ is feasible to \eqref{linear-linear BP} if and only if the value of \rm{Infeasibility} equals to zero.

Note that, since the lower-level problem is linear, \eqref{linear-linear BP} can be equivalently transformed into the following single-level problems:
\begin{eqnarray}\label{linear-linear MPCC}	
{\rm(MPCC)}\quad\min&&c_1^T x+c_2^T y \nonumber\\
\mbox{s.t.}&&A_1x \leq b_1, ~d_2+B_2^T u_1+u_2-u_3=0,\\
&&0\leq u_1 \perp A_2x+B_2y- b_2\leq 0,\nonumber\\
&&0\leq u_2 \perp y-b_u\leq 0, ~0\leq u_3 \perp y-b_l\geq 0. \nonumber
\end{eqnarray}
\begin{eqnarray}\label{linear-linear WDP}	
{\rm(WDP)}\quad\min &&c_1^T x+c_2^T y \nonumber\\
\mbox{s.t.}&&A_1x \leq b_1, ~d_2+B_2^T u_1+u_2-u_3=0, \\
&&d_2^T y-d_2^T z-u_1^T (A_2x+B_2z-b_2)-u_2^T (z-b_u)-u_3^T(b_l-z)\leq 0, \nonumber\\
&&A_2x+B_2y \leq b_2,~b_l\leq y\leq b_u,~ u_1,u_2,u_3\geq0.\nonumber
\end{eqnarray}
\begin{eqnarray}\label{linear-linear MDP}	
{\rm(MDP)}\quad\min&&c_1^T x+c_2^T y \nonumber\\
\mbox{s.t.}&&A_1x \leq b_1, ~d_2+B_2^T u_1+u_2-u_3=0, ~d_2^T y-d_2^T z \leq0,\\
&& u_1^T (A_2x+B_2z-b_2)+u_2^T (z-b_u)+u_3^T (b_l-z)\geq 0, \nonumber\\
&&A_2x+B_2y \leq b_2,~b_l\leq y\leq b_u,~ u_1,u_2,u_3\geq0,\nonumber
\end{eqnarray}
\begin{eqnarray}\label{linear-linear eMDP}
{\rm(eMDP)}\quad\min &&c_1^T x+c_2^T y \nonumber\\
\mbox{s.t.}&&A_1x \leq b_1, ~d_2+B_2^T u_1+u_2-u_3=0, ~d_2^T y-d_2^T z \leq0,\qquad\qquad\\
&& u_1 \circ(A_2x+B_2z-b_2)\geq 0,~u_2\circ (z-b_u)\geq 0, ~u_3\circ(b_l-z)\geq 0,\nonumber\\
&&A_2x+B_2y \leq b_2,~b_l\leq y\leq b_u,~ u_1,u_2,u_3\geq0.\nonumber
\end{eqnarray}
Similarly as in \cite{Li2022novel}, we first attempted to solve \eqref{linear-linear MPCC}-\eqref{linear-linear eMDP} directly and, as expected, the numerical results were very unsatisfactory. In fact, among 150 tested problems, none of them satisfied the feasibility accuracy to \eqref{linear-linear BP} by these direct methods; see the right-side columns of Tables C.1-C.15 in Appendix C. The only reason was that not only MPCC, but also WDP, MDP and eMDP have combinatorial structures. We need to employ some relaxation technique to improve success efficiency in solving bilevel programs.

\subsection{Relaxation schemes for MDP {and eMDP}}
\label{Comparison relaxation MDP}

Note that \eqref{linear-linear BP} has no equality constraints at lower-level. Compared with WDP, the key constraint $f(x, y) - f(x, z) - u^T g(x, z) \leq 0$ is decomposed into two constraints $f(x, y) - f(x, z)\leq 0$ and $u^T g(x, z) \geq 0$ in MDP, while the latter constraint is further decomposed into $u\circ g(x, z)\geq 0$ in eMDP. Since eMDP and MDP are structurally similar, this part is described in terms of MDP and, understandably, the analysis is generally suitable to eMDP.
	
There are three relaxation schemes for MDP, that is, the relaxation technique is applied to one of two constraints or to both of them. In the case of \eqref{linear-linear BP}, the relaxation problems are respectively
\begin{eqnarray*}
\mathrm{MDP_1}(t)\qquad\min&&c_1^T x+c_2^T y \nonumber\\
\mbox{s.t.}&&A_1x \leq b_1, ~d_2+B_2^T u_1+u_2-u_3=0, ~d_2^T y-d_2^T z\leq t,\\
&&u_1^T(A_2x+B_2z-b_2)+u_2^T(z-b_u)+u_3^T(b_l-z)\geq 0, \nonumber\\
&&A_2x+B_2y \leq b_2,~b_l\leq y\leq b_u,~ u_1,u_2,u_3\geq0,\nonumber
\end{eqnarray*}
\begin{eqnarray*}
\mathrm{MDP_2}(t)\qquad\min &&c_1^T x+c_2^T y \nonumber\\
\mbox{s.t.}&&A_1x \leq b_1, ~d_2+B_2^T u_1+u_2-u_3=0, ~d_2^T y-d_2^T z\leq 0,\\	&&u_1^T(A_2x+B_2z-b_2)+u_2^T(z-b_u)+u_3^T(b_l-z)\geq -t, \nonumber\\
&&A_2x+B_2y \leq b_2,~b_l\leq y\leq b_u,~ u_1,u_2,u_3\geq0,\nonumber
\end{eqnarray*}
\begin{eqnarray*}
\mathrm{MDP_3}(t)\qquad\min&&c_1^T x+c_2^T y \nonumber\\
\mbox{s.t.}&&A_1x \leq b_1, ~d_2+B_2^T u_1+u_2-u_3=0, ~d_2^T y-d_2^T z\leq t,\\	&&u_1^T(A_2x+B_2z-b_2)+u_2^T(z-b_u)+u_3^T(b_l-z)\geq -t, \nonumber\\
&&A_2x+B_2y \leq b_2,~b_l\leq y\leq b_u,~ u_1,u_2,u_3\geq0,\nonumber
\end{eqnarray*}
where $t>0$ is a relaxation parameter. One natural question is which is more effective. To answer this question, we implemented Algorithm \ref{alg:MDP-Relaxed} given below to evaluate the above relaxation schemes by solving the tested problems.

\begin{algorithm}
	\caption{}
	\label{alg:MDP-Relaxed}
	
	\textbf{Step 0}  Choose an initial point $\tilde{x}^0$ such that $A_1\tilde{x}^0 \leq b_1$,
	an initial relaxation parameter $t_0>0$, an update parameter $\sigma\in(0,1)$, and accuracy tolerances $\{\epsilon_{\rm \scriptscriptstyle SQP},\epsilon_{\rm r}\}$. Set $k:=0$.
	
	\textbf{Step 1}
	Solve the lower-level problem in \eqref{linear-linear BP} with $x=\tilde{x}^k$ by \emph{linprog} to get $\tilde{y}^k$ and $\tilde{u}^k$. Set $\tilde{w}^k:=(\tilde{x}^k,\tilde{y}^k,\tilde{y}^k,\tilde{u}^k)$.
	
	\textbf{Step 2} Solve MDP$(t_k)$ by \emph{fmincon} with $\tilde{w}^k$ as  starting point and $\epsilon_{\rm \scriptscriptstyle SQP}$ as accuracy tolerance to obtain an iterative $w^k:=(x^k,y^k,z^k,u^k)$. If $w^k$ satisfies some termination criterion, stop. Otherwise, go to Step 3.
	
	\textbf{Step 3}
	Set $\tilde{x}^{k+1}:={x}^k$ and $t_{k+1}:=\max\{\sigma t_k, \epsilon_{\rm r}\}$. Return to Step 1 with $k:=k+1$.
\end{algorithm}

The point $\tilde{w}^k$ generated by Algorithm \ref{alg:MDP-Relaxed} is obviously feasible to \eqref{linear-linear MDP} as long as the lower-level problem in \eqref{linear-linear BP} with $x=\tilde{x}^k$ is solvable. In the case where the lower-level problem had no solution, we chose $\tilde{y}^k=0$ and $\tilde{u}^k=0$. We set $t_0=0.1$, $\sigma=0.5$, $\epsilon_{\rm r}=10^{-8}$, and $\epsilon_{\rm \scriptscriptstyle SQP}=10^{-16}$. In addition, the termination criterion in Step 2 was either $t_k\leq\epsilon_{\rm r}$ or
\begin{eqnarray*}
&& |d_2^T y^k-d_2^T z^k|\leq\epsilon_{\rm r}~~{\rm for~MDP_1}(t_k), \\
&&|(u_1^k)^T(A_2x^k+B_2z^k-b_2)+(u_2^k)^T(z^k-b_u)+(u_3^k)^T(b_l-z^k)|\leq \epsilon_{\rm r}~~{\rm for~MDP_2}(t_k),\\
&&|d_2^T y^k-d_2^T z^k|\leq\epsilon_{\rm r},~ |(u_1^k)^T(A_2x^k+B_2z^k-b_2)+(u_2^k)^T(z^k-b_u)+(u_3^k)^T(b_l-z^k)|\leq \epsilon_{\rm r}~~{\rm for~MDP_3}(t_k).
\end{eqnarray*}

Numerical results for the tested problems by three relaxation schemes are shown in Tables C.1-C.15 in Appendix C, where ``\#'' represents the example serial number, ``ObjVal'' represents the upper-level objective value at the approximation solution, and ``Time'' represents the average CPU time (in seconds) by running 3 times to solve the same problem.

Comparison results of three relaxation schemes for MDP are summarized in Table \ref{tab:Comparison of three relaxation schemes}, where ``Number of feasible cases" denotes the number of problems satisfying Infeasibility~$<10^{-3}$ among 50 tested problem and ``Average time'' represents the average CPU time for 50 problems.

\begin{table}[htbp]
	\centering
	\caption{Comparison of three relaxation schemes MDP$_1$, MDP$_2$, and MDP$_3$}
	\vspace{2mm}
	\begin{tabular}{c|ccc}
		\hline	
		Dimension  & Scheme & Number of feasible cases & {Average time}\\[1mm]
		\hline
		\multirow{3}[0]{*}{$n=20, l=30,  m=60, p=50$}
		& MDP$_1$  & 18    & 17.71  \\[1mm]
		& MDP$_2$  & 8     & 49.71  \\[1mm]
		& MDP$_3$  & 5     & 59.52  \\[1mm]
		\cmidrule(r){1-4}
		\multirow{3}[0]{*}{$n=20, l=30, m=100, p=110$}
		& MDP$_1$  & 19    & 80.30  \\[1mm]
		& MDP$_2$  & 16    & 149.82 \\[1mm]
		& MDP$_3$  & 12    & 161.81 \\[1mm]
		\cmidrule(r){1-4}
		\multirow{3}[0]{*}{$n=20, l=30, m=160, p=150$}
		& MDP$_1$  & 20    & 409.05 \\[1mm]
		& MDP$_2$  & 7     & 508.79 \\[1mm]
		& MDP$_3$  & 12    & 480.77 \\
		\hline
	\end{tabular}
	\label{tab:Comparison of three relaxation schemes}%
\end{table}

It can be seen from Table \ref{tab:Comparison of three relaxation schemes} that MDP$_1$ was notably better than the others in terms of the feasibility to the original bilevel programs and the average CPU time. Based on this observation, in order to compare the efficiency of the MDP approach with the MPCC and WDP approaches, we use the relaxation scheme MDP$_1(t)$ to approximate MDP in the next subsection. Since eMDP is similar to MDP, it is natural to use
\begin{eqnarray*}
\mathrm{eMDP}(t)\qquad\min&&c_1^T x+c_2^T y \\
\mbox{s.t.}&&A_1x \leq b_1, ~d_2+B_2^T u_1+u_2-u_3=0, ~d_2^T y-d_2^T z \leq t,\\
&& u_1 \circ(A_2x+B_2z-b_2)\geq 0,~u_2\circ (z-b_u)\geq 0, ~u_3\circ(b_l-z)\geq 0,\\
&&A_2x+B_2y \leq b_2,~b_l\leq y\leq b_u,~ u_1,u_2,u_3\geq0
\end{eqnarray*}
to approximate eMDP.

For completeness, we give a comprehensive convergence analysis for the relaxation {schemes MDP$_1$ and $\mathrm{eMDP}(t)$} in {Online Appendix}.

\subsection{Comparison of {MPCC, WDP, MDP, and eMDP}}

As stated in Subsection \ref{Test examples}, direct methods were very unsatisfactory in solving the tested problems. In this subsection, we introduce our numerical experiments on relaxation methods for MDP, eMDP, MPCC, and WDP. Our purpose is to evaluate which is more effective.

Consider the linear bilevel program \eqref{linear-linear BP}. Given a relaxation parameter $t>0$, we use MDP$_1(t)$ to approximate the MDP \eqref{linear-linear MDP}, use eMDP$(t)$ to approximate the eMDP \eqref{linear-linear eMDP}, use
\begin{eqnarray*}
\mathrm{MPCC}(t)\qquad\min &&c_1^T x+c_2^T y \nonumber\\
\mbox{s.t.}&&A_1x \leq b_1,~d_2+B_2^T u_1+u_2-u_3=0,\nonumber\\
&&u_1^T (A_2x+B_2y-b_2)+u_2^T (y-b_u)+u_3^T (b_l-y)\leq t,\qquad\qquad\nonumber\\
&&A_2x+B_2y \leq b_2,~b_l\leq y\leq b_u,~ u_1,u_2,u_3\geq0\nonumber
\end{eqnarray*}
to approximate the MPCC \eqref{linear-linear MPCC}, and use
\begin{eqnarray*}
\mathrm{WDP}(t)\qquad\min &&c_1^T x+c_2^T y \nonumber\\
\mbox{s.t.}&&A_1x \leq b_1, ~d_2+B_2^{T} u_1+u_2-u_3=0, \\
&&d_2^T y-d_2^T z-u_1^T (A_2x+B_2z-b_2)-u_2^T (z-b_u)
	-u_3^T (b_l-z)\leq t, \nonumber\\
&&A_2x+B_2y \leq b_2,~b_l\leq y\leq b_u,~ u_1,u_2,u_3\geq0\nonumber
\end{eqnarray*}
to approximate the WDP \eqref{linear-linear WDP}.

In our experiments, we implemented Algorithm \ref{alg:MDP-Relaxed} to solve MDP by relaxation method. Moreover, by replacing MDP$(t_k)$ with eMDP$(t_k)$ and updating the termination criterion to either $t_k\leq\epsilon_{\rm r}$ or $|d_2^T y^k-d_2^T z^k|\leq\epsilon_{\rm r}$ in Step 2, Algorithm \ref{alg:MDP-Relaxed} was used to solve eMDP  by relaxation method. In addition, by replacing MDP$(t_k)$ with WDP$(t_k)$ and updating the termination criterion to either $t_k\leq\epsilon_{\rm r}$ or
$$|d_2^T y^k-d_2^T z^k-(u_1^k)^T (A_2x^k+B_2z^k-b_2)-(u_2^k)^T (z^k-b_u)-(u_3^k)^T (b_l-z^k)|\leq\epsilon_{\rm r}$$
in Step 2, Algorithm \ref{alg:MDP-Relaxed} was used to solve WDP  by relaxation method. We made similar modification on Algorithm \ref{alg:MDP-Relaxed} to solve MPCC by relaxation method and particularly, the termination criterion was updated to either $t_k\leq\epsilon_{\rm r}$ or
$$|(u_1^k)^T (A_2x^k+B_2y^k-b_2)+(u_2^k)^T (y^k-b_u)+(u_3^k)^T (b_l-y^k)|\leq\epsilon_{\rm r}.$$ The settings of parameters were the same as in Subsection \ref{Comparison relaxation MDP}, that is, $t_0=0.1$, $\sigma=0.5$, $\epsilon_{\rm r}=10^{-8}$, and $\epsilon_{\rm \scriptscriptstyle SQP}=10^{-16}$.

\begin{table}[htbp]
	\centering
	\caption{Comparison of MDP {and eMDP} with MPCC and WDP }
	\vspace{2mm}
	\begin{tabular}{c|ccc}
		\hline
		Dimension & Method & Number of dominant cases & {Average time} \\[1mm]
		\hline
		\multirow{4}[0]{*}{$n=20, l=30, m=60, p=50$ }
		& MPCC  & 7   & 74.55  \\[1mm]
		& WDP   & 6   & 54.12  \\[1mm]
		& MDP   & 13  & 17.71  \\[1mm]
		& eMDP  & 28  & 27.65  \\[1mm]
		\cmidrule(r){1-4}
		\multirow{4}[0]{*}{$n=20, l=30, m=100, p=110$}
		& MPCC  & 6   & 145.05 \\[1mm]
		& WDP   & 12  & 107.69 \\[1mm]
		& MDP   & 15  & 80.30  \\[1mm]
		& eMDP  & 29  & 109.56  \\[1mm]
		\cmidrule(r){1-4}
		\multirow{4}[0]{*}{$n=20, l=30, m=160, p=150$}
		& MPCC  & 10  & 434.91 \\[1mm]
		& WDP   & 14  & 637.83 \\[1mm]
		& MDP   & 12  & 409.05 \\[1mm]
		& eMDP  & 36  & 117.63  \\
		\hline
	\end{tabular}
	\label{tab:Comparison of MPCC, WDP, MDP}
\end{table}

Numerical results by relaxation methods for the tested problems are shown in the left-side columns of Tables C.1-C.15 in Appendix C. Comparison results of four approaches are summarized in Table \ref{tab:Comparison of MPCC, WDP, MDP}, where the dominant cases mean that the corresponding approach obtained the best or one of the best approximation solutions satisfying Infeasibility~$<10^{-3}$ by comparing their final objective values. Note that {\rm{``Infeasibility"}} is an output result of the algorithms, which is the top priority criterion for measuring each algorithm's performance, rather than a termination condition. Speaking briefly, {\rm{``Infeasibility"}} is a criterion to measure the infeasibility of the approximation solutions obtained by each algorithm to the original linear bilevel program \eqref{linear-linear BP}.

From Table \ref{tab:Comparison of MPCC, WDP, MDP}, it can be observed that, at least in our tests, the MDP approach was better than the MPCC and WDP approaches, while the eMDP approach was far superior to all other three approaches. In particular, the total numbers of dominant cases for MPCC, WDP, MDP, and eMDP are 23, 32, 40, and 93, respectively. Moreover, in most cases, the MDP and eMDP approaches took shorter average CPU times, which may result in significant computational cost savings, especially for large-scale problems.

One main reason why the MDP and eMDP approaches performed better in our tests may be because eMDP, MDP and WDP have better properties than MPCC, while eMDP has tighter feasible region than MDP, and MDP has tighter feasible region than WDP so that the feasibility to the original bilevel program is easier to achieve. To some extent, this may be verified from the experimental results shown in Tables C.1-C.15 in Appendix C. In fact, the total numbers of feasible cases for MPCC, WDP, MDP, and eMDP are 32, 42, 57, and 128, respectively, which are not much different from the total numbers of dominant cases. In addition, as stated in Section \ref{intro}, MDP has a merit that the constraint $f(x, y) - f(x, z)\leq 0$ may be simplified when the lower-level objective is composed of a strictly monotone one-dimensional function $f(t)$ with $t=r(x,y)$. For such cases, introducing preprocessing procedure may improve the success efficiency of the MDP approach. We would like to leave this work together with some practical applications as a future task.

\medskip
\section{Conclusions}
\label{conclusion}
Motivated by the recent work \cite{Li2022novel}, we applied the lower-level Mond-Weir duality to present the new MDP reformulation for the bilevel program BP. Under mild assumptions, their equivalence was shown in globally or locally optimal sense. Similar to the WDP reformulation studied in \cite{Li2022novel}, MDP may also satisfy the MFCQ at its feasible points, which is different from the well-known MPCC reformulation that does not satisfy the MFCQ at any feasible point. Moreover, we investigated the relations among these reformulations. We further extend the MDP reformulation to present the new eMDP reformulation, which has similar properties to MDP. Furthermore, in order to compare the new MDP and eMDP approaches with the MPCC and WDP approaches, we designed a procedure to generate 150 tested problems randomly and our numerical experiments showed that, at least in our tests, MDP has evident advantages over MPCC and WDP in terms of feasibility to the original bilevel programs, success efficiency, and average CPU time, while eMDP is far superior to all other three reformulations.

MDP, eMDP, and WDP not only have theoretical advantages, but also have computational advantages over MPCC. This inspires us to further develop more effective algorithms for bilevel programs by means of lower-level duality. In particular, we may expect to apply some active-set identification techniques, like in \cite{Lin2006hybrid}, to the separate constraints $u\circ g(x,z)\geq0$ in eMDP to develop some new algorithms.
In addition, we may apply the new approaches to practical problems in economy, management science, and so on.

\bigskip
\bigskip
\vspace{4pt} \noindent{\bf Acknowledgement.}
The authors are grateful to Prof. Masao Fukushima whose helpful comments and suggestions have led to much improvement of the paper. All authors contributed equally to this work and are listed in alphabetical order.

%
%

\bigskip

\begin{appendices}

\setcounter{equation}{0}
\renewcommand\theequation{A.\arabic{equation}}
\setcounter{thm}{0}
\renewcommand\thethm{A.\arabic{thm}}
\setcounter{defi}{0}
\renewcommand\thedefi{A.\arabic{defi}}
\setcounter{table}{0}
\renewcommand\thetable{B.\arabic{table}}

\section*{Appendix A: Duality theorem for $\mathrm{P}_x$ and $\mathrm{eMD}_x$}
\begin{thm}\label{eMD-dualityth}
Suppose that, for any given $x\in X$,  $u\in \mathbb{R}^p_+$, and $v\in \mathbb{R}^q$, $f(x,\cdot)$ is pseudoconvex and $u^{T} g(x, \cdot)+v^{T} h(x, \cdot)$ is quasiconvex over $Y(x)\cup \mathcal{Z}'(x,u,v)$. Then, the weak duality holds, i.e.,
\begin{eqnarray}\label{eMDweakduality}
\min\limits_{y\in Y(x)} f(x,y) \geq \max\limits_{(z, u, v)\in M'(x)} f(x, z).
\end{eqnarray}
\noindent Moreover, if $\mathrm{P}_x$ satisfies the Guignard constraint qualification at some solution $y_x \in {S(x)}$, the strong duality holds, i.e.,
\begin{eqnarray*}
\min\limits_{y\in Y(x)} f(x,y) = f(x,y_x)
			= \max\limits_{(z, u, v)\in M'(x)} f(x,z).
\end{eqnarray*}
\end{thm}

\begin{proof}
We first show the weak duality. In fact, for any $y\in Y(x)$ and $(z,u,v)\in M'(x)$, we have
\begin{eqnarray}\label{eMDP-th-1}
u^T g(x, y)+v^T h(x, y)-(u^T g(x, z)+v^T h(x, z))\leq0.
\end{eqnarray}
Since $u^{T} g(x, \cdot)+v^T h(x, \cdot)$ is quasiconvexity,  \eqref{eMDP-th-1} implies
\begin{eqnarray}\label{eMDP-th-2}
(y-z)^T(\nabla_zg(x, z)u+\nabla_zh(x, z)v)\leq 0.
\end{eqnarray}
Since $\nabla_z f(x, z) + \nabla_z g(x, z) u + \nabla_z h(x, z) v= 0$, it follows that
\begin{eqnarray}\label{eMDP-th-3}
(y-z)^T(\nabla_zf(x, z)+\nabla_zg(x, z)u + \nabla_z h(x, z) v )=0.
\end{eqnarray}
Based on \eqref{eMDP-th-2} and \eqref{eMDP-th-3}, we have
\begin{eqnarray*}
(y-z)^T\nabla_zf(x, z)=-(y-z)^T(\nabla_zg(x, z)u+ \nabla_z h(x, z) v ) \geq0.
\end{eqnarray*}
Since $f(x,\cdot)$ is pseudoconvex, $f(x,y) \geq f(x, z)$ holds.
By the arbitrariness of $y\in Y(x)$ and $(z,u,v)\in M'(x)$, we have
\begin{eqnarray}\label{eMDweakduality}
\min\limits_{y\in Y(x)} f(x,y)  \geq \max\limits_{(z, u,v)\in M'(x)} f(x, z).
\end{eqnarray}

Now we show the strong duality. Since $\mathrm{P}_x$ satisfies the Guignard constraint qualification at $y_x \in {S}(x)$, there exists $(u_x, v_x )\in\mathbb{R}^{p+q}$ such that
\begin{eqnarray*}
\nabla_y L(x,y_x,u_x,v_x)=0,~ u_x^T g(x,y_x)=0, ~u_x\geq 0,~g(x,y_x)\leq0,~h(x,y_x)=0,
\end{eqnarray*}
which implies $(y_x,u_x,v_x)\in M'(x)$.
Then, we have
\begin{eqnarray*}
\min\limits_{y\in Y(x)} f(x,y)  \geq \max\limits_{(z, u,v)\in M'(x)} f(x, z)\geq f(x,y_x).
\end{eqnarray*}
Since  $y_x \in {S}(x)$, the above inequalities are all equalities.
As a result, we have
\begin{eqnarray*}
\min\limits_{y\in Y(x)} f(x,y) = f(x,y_x)= \max\limits_{(z, u,v)\in M'(x)} f(x,z).
\end{eqnarray*}
This completes the proof.
\end{proof}

\setcounter{equation}{0}
\renewcommand\theequation{B.\arabic{equation}}
\setcounter{thm}{0}
\renewcommand\thethm{B.\arabic{thm}}
\setcounter{defi}{0}
\renewcommand\thedefi{B.\arabic{defi}}
\setcounter{table}{0}

\section*{Appendix B: Convergence analysis for relaxation schemes}
Although MDP and eMDP may satisfy the MFCQ at their feasible points, numerical experiments reveal that {they are} still difficult to solve directly. In this part, we give convergence analysis for the following relaxation schemes:
\begin{eqnarray*}
\mathrm{MDP}(t)\qquad\min&&F(x, y) \nonumber \\
		\mbox{\rm s.t.}&&(x,y)\in \Omega,~g(x, y) \le 0, ~h(x, y)=0,\\
		&&f(x, y) - f(x, z)\leq t,~ u^T g(x, z) + v^T h(x,z)\geq0,\nonumber\\
		&&\nabla_z f(x, z) + \nabla_z g(x, z) u + \nabla_z h(x, z) v = 0, ~u\geq 0\nonumber
\end{eqnarray*}
and
\begin{eqnarray*}
\mathrm{eMDP}(t)\qquad\min&&F(x, y) \nonumber \\
\mbox{\rm s.t.} &&(x,y)\in \Omega,~g(x, y) \le 0,~h(x, y) = 0,\\
&&f(x, y) - f(x, z)\leq t,~ u\circ g(x, z) \geq0,~ v\circ h(x, z) =0,\nonumber\\
&&\nabla_z f(x, z) + \nabla_z g(x, z) u + \nabla_z h(x, z) v = 0, ~u\geq 0,\nonumber
\end{eqnarray*}
where $t>0$ is a relaxation parameter. Other relaxation schemes introduced in Subsection 7.2 can be analyzed similarly. The following convergence result related to globally optimal solutions is evident.

\begin{thm}
Let $\{t_k\}\downarrow0$ and $(x^k,y^k,z^k,u^k,v^k)$ be a globally optimal solution of ${\mathrm{MDP}(t_k)}$ ({\it or}  ${\mathrm{eMDP}(t_k)}$) for each $k$. Then, every accumulation point of $\{(x^k,y^k,z^k,u^k,v^k)\}$ is globally optimal to ${\rm MDP}$ ({\it or} ${\rm eMDP}$).
\end{thm}

Since ${\mathrm{MDP}(t)}$ and ${\mathrm{eMDP}(t_k)}$ are usually nonconvex optimization  problems, it is necessary to study the limiting properties of their stationary points. To this end, the following constraint qualification is needed.

\begin{defi}\rm  \cite{Qi2000constant}
\label{CPLD}
We say that the problem {\rm P} in Section 2 satisfies the constant positive linear dependence (CPLD) at a feasible point $\bar{y}$ if, for any $I_0\subseteq \{ i : g_i(\bar{y})=0\}$ and $J_0\subseteq \{1,\cdots, q\}$ such that the gradients $\{\nabla g_i(\bar{y}), i\in I_0\}\cup\{\nabla h_j(\bar{y}), j\in J_0\}$ are positive-linearly dependent, there exists a neighborhood $\bar{U}$ of $\bar{y}$ such that $\{\nabla g_i(y), i\in I_0; ~\nabla h_j(y), j\in J_0\}$ are linearly dependent for any $y\in \bar{U}$.
\end{defi}

\begin{thm}\label{CPLDth}
If $\mathrm{MDP}(t)$ ({\it or} $\mathrm{eMDP}(t)$) satisfies the CPLD at a feasible point $(\bar{x},\bar{y},\bar{z},\bar{u},\bar{v})$,
there exists a neighborhood $\bar{U}$ of $(\bar{x},\bar{y},\bar{z},\bar{u},\bar{v})$ such that $\mathrm{MDP}(t)$ ({\it or} $\mathrm{eMDP}(t)$) satisfies the CPLD at any feasible point in $\bar{U}$.
\end{thm}
	
\begin{proof}
It is sufficient to show the case of MDP, since the case of eMDP can be shown in a similar way.
For simplicity, we take the upper constraint $(x,y)\in \Omega$ away from $\mathrm{MDP}(t)$ and denote by $\bar{w}:=(\bar{x},\bar{y},\bar{z},\bar{u},\bar{v})$, $w:=(x,y,z,u,v)$. We further denote by $G_i(w) ~(i=0,1,\cdots,2p+1)$ all inequality constraint functions 
and by $H_i(w) ~(i=1,\cdots,q+m)$ all equality constraint functions.
	
Suppose by contradiction that the conclusion is wrong. That is, there is a sequence $\{w^k\}$ of feasible points tending to $\bar{w}$ such that $\mathrm{MDP}(t)$ does not satisfy the CPLD at each $w^k$. Violation of CPLD means that there exist $I_0^k\subseteq \{i : G_i(w^k)=0\}$ and $J_0^k\subseteq \{1,\cdots,q+m\}$  such that $\{\nabla G_i(w^k), i\in I_0^k\}\cup\{\nabla H_j(w^k), j\in J_0^k\}$ are positive-linearly dependent, but linearly independent in some points arbitrarily close to $w^k$. Note that the number of subsets of the index sets $\{i : G_i(w^k)=0\}$ and $\{1,\cdots,q+m\}$
is finite. Taking a subsequence if necessary, we may assume $I_0^k\equiv I_0$ and $J_0^k\equiv J_0$. Therefore, there is another sequence $\{\tilde{w}^k\}$ tending to $\bar{w}$ such that $\{\nabla G_i(\tilde{w}^k), i\in I_0; ~\nabla H_j(\tilde{w}^k), j\in J_0\}$ is linearly independent.
	
If $\{\nabla G_i(\bar{w}), i\in I_0\}\cup\{\nabla H_j(\bar{w}), j\in J_0\}$ is positive-linearly dependent, by the CPLD assumption, there exists a neighborhood which remains the linear dependence, which contradicts the existence of $\{\tilde{w}^k\}$. If $\{\nabla G_i(\bar{w}), i\in I_0\}\cup\{\nabla H_j(\bar{w}), j\in J_0\}$ is positive-linearly independent, by Proposition 2.2 in \cite{Qi2000constant},
this property should remain in a neighborhood, which contradicts the existence of $\{w^k\}$. As a result, the conclusion is true. This completes the proof.
\end{proof}
	
Now, we give the main convergence result for the relaxation schemes MDP$(t)$ {and $\mathrm{eMDP}(t)$}.
	
\begin{thm}
Let $\{t_k\}\downarrow0$ and $(x^k,y^k,z^k,u^k,v^k)$ be a KKT point of ${\mathrm{MDP}(t_k)}$ ({\it or} $\mathrm{eMDP}(t_k)$) for each $k$. Assume that $(\bar{x},\bar{y},\bar{z},\bar{u},\bar{v})$ is an accumulation point of $\{(x^k,y^k,z^k,u^k,v^k)\}$. If {\rm MDP} ({\it or} {\rm eMDP}) satisfies the CPLD at $(\bar{x},\bar{y},\bar{z},\bar{u},\bar{v})$, then $(\bar{x},\bar{y},\bar{z},\bar{u},\bar{v})$ is a KKT point of ${\rm MDP}$ ({\it or} {\rm eMDP}).
\end{thm}

\begin{proof}
As in the above proof, it is sufficient to show the case of MDP. For simplicity, we take the upper constraint $(x,y)\in \Omega$ away from MDP and $\mathrm{MDP}(t)$, and denote
\begin{eqnarray*}		&&\bar{w}:=(\bar{x},\bar{y},\bar{z},\bar{u},\bar{v}), ~~~w^k:=(x^k,y^k,z^k,u^k,v^k), ~~~w:=(x,y,z,u,v),\\
		&&F_0(w):=F(x,y), ~~~G_0(w):=f(x,y)-f(x, z), ~~~G_1(w):=-u^Tg(x,z)-v^Th(x,z),\\
		&&G_i(w):=g_{i-1}(x,y) ~(2\le i\le p+1),  ~~~G_i(w):=-u_{i-p-1} ~(p+2\le i\le 2p+1),\\
		&&H_i(w):=h_i(x,y) ~(1\le i\le q),  ~~H_i(w):=\nabla_{z_{i-q}}L(x,z,u,v) ~(q+1\le i\le q+m).
\end{eqnarray*}

Then, MDP and $\mathrm{MDP}(t_k)$ become
\begin{eqnarray}\label{app-MDP}
\min&&F_0(w) \nonumber \\
\mbox{\rm s.t.} && G_i(w) \leq 0 ~(i=0,1,\cdots,2p+1),\\
&&H_i(w) =0 ~(i=1,\cdots,q+m) \nonumber
\end{eqnarray}
and
\begin{eqnarray}\label{app-MDPt}
\min &&F_0(w) \nonumber \\
\mbox{\rm s.t.} && G_0(w) \leq t_k, \nonumber\\
&& G_i(w) \leq 0 ~(i=1,\cdots,2p+1),\\
&&H_i(w) =0~(i=1,\cdots,q+m). \nonumber
\end{eqnarray}
Obviously, $\bar{w}$ is a feasible point of MDP and, without loss of generality, we may assume that the whole sequence $\{w^k\}$ converges to $\bar{w}$.
Since $w^k$ is a KKT point of \eqref{app-MDPt}, it follows from Lemma A.1 in \cite{Steffensen2010new}
that there exists $(\lambda_0^k,\lambda^k,\mu^k) \in \mathbb{R}\times\mathbb{R}^{2p+1}\times\mathbb{R}^{q+m}$ such that
\begin{eqnarray}\label{app-MDPt1}
&&\nabla F_0(w^k)+\lambda_0^k\nabla G_0(w^k)+\nabla G(w^k)\lambda^k+\nabla H(w^k)\mu^k=0,
		\nonumber\\
&&\lambda_0^k(G_0(w^k)-t_k)=0, ~~\lambda_0^k\geq0,
\\
&&G(w^k)^T\lambda^k=0, ~~\lambda^k\geq0,\nonumber
\end{eqnarray}
and the gradients
\begin{eqnarray}\label{independent}
\{\nabla G_i(w^k): 
		\lambda_i^k>0,~0\le i\le 2p+1 \}\cup\{\nabla H_i(w^k): 
\mu_i^k\not=0\}
\end{eqnarray}
are linearly independent. To show that $\bar{w}$ is a KKT point of \eqref{app-MDP}, it is sufficient to show the boundedness of the sequence of multipliers $\{\lambda_0^k,\lambda^k,\mu^k\}$.

Assume by contradiction that $\{\lambda_0^k,\lambda^k,\mu^k\}$ is unbounded. Taking a subsequence if necessary, we assume
\begin{eqnarray}\label{0}		\lim_{k\to\infty}\frac{(\lambda_0^k,\lambda^k,\mu^k)}
		{\|(\lambda_0^k,\lambda^k,\mu^k)\|} = (\lambda_0^*,\lambda^*,\mu^*). 
\end{eqnarray}
In particular, for every $k$ sufficiently large, we have
\begin{eqnarray}\label{supp}
\lambda_i^*>0 \ \Rightarrow \ \lambda_i^k>0, \quad \mu_i^*\not=0 \ \Rightarrow \ \mu_i^k\not= 0.
\end{eqnarray}
Dividing by $\|(\lambda_0^k,\lambda^k,\mu^k)\|$ in \eqref{app-MDPt1} and taking a limit, we have
\begin{eqnarray}\label{app-MDPt2}
\lambda_0^*\nabla G_0(\bar{w})+\sum_{\lambda_i^*>0}\lambda_i^*\nabla G_i(\bar{w})+\sum_{\mu_i^*\not=0}\mu_i^*\nabla H_i(\bar{w})=0, ~~\lambda_0^*\geq0.
\end{eqnarray}
Since $\|(\lambda_0^*,\lambda^*,\mu^*)\|=1$ by \eqref{0}, \eqref{app-MDPt2} implies that the gradients
\begin{eqnarray*}
\{\nabla G_i(\bar{w}): \lambda_i^*>0, 0\le i\le 2p+1\}\cup\{\nabla H_i(\bar{w}): \mu_i^*\not=0\}
\end{eqnarray*}
are positive-linearly dependent. Since problem \eqref{app-MDP} satisfies the CPLD at $\bar{w}$, the gradients
\begin{eqnarray*}
\{\nabla G_i(w): \lambda_i^*>0, 0\le i\le 2p+1\}\cup\{\nabla H_i(w): \mu_i^*\not=0\}
\end{eqnarray*}
are linearly dependent in a neighborhood. This, together with \eqref{supp}, contradicts the linear independence of the vectors in \eqref{independent}.
Consequently, $\{\lambda_0^k,\lambda^k,\mu^k\}$ is bounded.
This completes the proof.
\end{proof}

\setcounter{table}{0}
\renewcommand\thetable{C.\arabic{table}}
\section*{Appendix C: Numerical results}

\begin{table}[htbp]
\renewcommand\arraystretch{1}
  \centering
  \caption{Numerical results of three relaxation schemes for MDP with $\{n =20,l=30,m=60,p=50\}$}
  \tabcolsep=1.5mm
\begin{tabular}{llllllllllll}
    \hline
    & \multicolumn{3}{c}{MDP$_1$} &
    & \multicolumn{3}{c}{MDP$_2$} &
    & \multicolumn{3}{c}{MDP$_3$} \\
    \cmidrule(r){2-4}\cmidrule(r){6-8}\cmidrule(r){10-12}
    \#    & ObjVal & Infeasibility & Time  &
          & ObjVal & Infeasibility & Time  &
          & ObjVal & Infeasibility & Time  \\
    \hline

1  & 3.21     & 5.24     & 47.29  &
   & -15.11   & 1.66     & 90.91  &
   & 40.25    & 2.95     & 101.56 \\

2  & -16.50   & 0.96     & 18.81  &
   & -14.64   & 9.76     & 147.59 &
   & 24.59    & 2.37     & 46.85  \\

3  & -54.60   & 2.16e-04 & 4.87   &
   & -54.60   & 2.16e-04 & 5.33   &
   & -21.08   & 2.81     & 42.06  \\

4  & -28.70   & 1.01e-08 & 2.90   &
   & -5.06    & 0.25     & 19.43  &
   & -20.62   & 6.13     & 42.31  \\

5  & -1.80    & 7.03e-13 & 1.95   &
   & 25.80    & 3.14     & 147.19 &
   & -42.74   & 7.19     & 76.64  \\

6  & 10.86    & 1.66     & 10.28  &
   & -34.82   & 0.48     & 29.33  &
   & -17.16   & 0.22     & 7.40   \\
\hline
    \end{tabular}%
  \label{tab:three relaxation schemes for MDP(m=60,p=50)}
\end{table}%

\begin{table}[htbp]
\renewcommand\arraystretch{1}
  \centering
  \caption{Numerical results of three relaxation schemes for MDP with $\{n=20,l=30,m=100,p=110\}$}
  \tabcolsep=1.5mm
    \begin{tabular}{llllllllllll}
     \hline
    & \multicolumn{3}{c}{MDP$_1$} &
    & \multicolumn{3}{c}{MDP$_2$} &
    & \multicolumn{3}{c}{MDP$_3$} \\
    \cmidrule(r){2-4}\cmidrule(r){6-8}\cmidrule(r){10-12}
   \#    & ObjVal & Infeasibility & Time  &
         & ObjVal & Infeasibility & Time  &
         & ObjVal & Infeasibility & Time \\
   \hline

51 & -55.75  & 8.55e-13 & 13.41  &
   & -43.66  & 8.35e-09 & 10.96  &
   & -43.66  & 8.35e-09 & 12.24  \\

52 & -22.64  & 2.36e-05 & 41.38  &
   & 2.99    & 1.88e-05 & 50.03  &
   & 24.39   & 0.05     & 63.49  \\

53 & -42.35  & 2.12     & 146.52 &
   & -39.97  & 0.20     & 91.40  &
   & -43.21  & 1.07e-12 & 8.06   \\

54 & 30.74   & 2.11     & 156.74 &
   & 24.77   & 2.50     & 182.76 &
   & 1.66    & 1.16     & 40.64  \\

55 & -5.64   & 7.84e-13 & 8.48   &
   & -5.64   & 7.84e-13 & 8.41   &
   & -5.64   & 7.84e-13 & 8.12   \\

56 & -10.26  & 2.60     & 257.67 &
   & -83.59  & 24.78    & 201.80 &
   & -0.97   & 1.06e-04 & 25.68  \\
\hline
    \end{tabular}%
  \label{tab:three relaxation schemes for MDP(m=100,p=110)}
\end{table}%

\begin{table}[htbp]
\renewcommand\arraystretch{1}
  \centering
  \caption{Numerical results of three relaxation schemes for MDP with $\{n=20,l=30,m=160,p=150\}$}
  \tabcolsep=1.5mm
    \begin{tabular}{llllllllllll}
    \hline
    & \multicolumn{3}{c}{MDP$_1$} &
    & \multicolumn{3}{c}{MDP$_2$} &
    & \multicolumn{3}{c}{MDP$_3$} \\
     \cmidrule(r){2-4}\cmidrule(r){6-8}\cmidrule(r){10-12}
   \#    & ObjVal & Infeasibility & Time  &
         & ObjVal & Infeasibility & Time  &
         & ObjVal & Infeasibility & Time  \\
   \hline

101 & -2.82   & 2.89     & 500.21  &
    & 8.90    & 15.40    & 372.11  &
    & -50.97  & 19.61    & 219.53  \\

102 & 34.13   & 8.56e-09 & 35.47   &
    & -0.20   & 2.72     & 394.06  &
    & -78.73  & 2.07     & 946.63  \\

103 & -28.62  & 6.74e-04 & 107.01  &
    & 19.71   & 0.93     & 770.67  &
    & -10.03  & 3.17     & 632.37  \\

104 & -22.94  & 2.44     & 792.97  &
    & 4.86    & 0.25     & 974.11  &
    & -51.12  & 4.97     & 1664.67 \\

105 & -33.90  & 1.81     & 348.45  &
    & -63.54  & 3.54     & 533.71  &
    & -18.62  & 2.09e-05 & 166.49  \\

106 & -45.69  & 4.17     & 74.42   &
    & -36.25  & 0.37     & 176.91  &
    & -23.50  & 7.04     & 312.28  \\
\hline
    \end{tabular}%
  \label{tab:three relaxation schemes for MDP(m=160,p=150)}
\end{table}%

\begin{table}[htbp]
\renewcommand\arraystretch{1}
  \centering
  \caption{Numerical results for MPCC with $\{n=20, l=30, m=60, p=50\}$}
  \tabcolsep=5mm
    \begin{tabular}{llllllll}
    \hline
    & \multicolumn{3}{c}{Relaxation method} &
    & \multicolumn{3}{c}{Direct method} \\
    \cmidrule(r){2-4}\cmidrule(r){6-8}
    \#    & ObjVal  & Infeasibility & Time  &
          & ObjVal  & Infeasibility & Time  \\
    \hline

1  & -19.02        & 16.97      & 124.25    &
   & -0.05         & 114.17     & 5.85      \\

2  & -4.37         & 9.11       & 36.79     &
   & 0             & 130.53     & 5.20      \\

3  & -30.06        & 12.41      & 183.81    &
   & -0.02         & 135.52     & 5.60      \\

4  & -46.63        & 3.58       & 8.98      &
   & -3.04e-06     & 138.53     & 8.74      \\

5  & -1.80         & 7.03e-13   & 7.67      &
   & 0             & 148.46     & 5.63      \\

6  & -26.73        & 23.59      & 114.98    &
   & -1.02e-03     & 154.25     & 12.68     \\
\hline
    \end{tabular}%
  \label{tab:numerical results for MPCC (m=60, p=50)}%
\end{table}%

\begin{table}[htbp]
\renewcommand\arraystretch{1}
  \centering
  \caption{Numerical results for MPCC with $\{n=20, l=30, m=100, p=110\}$}
  \tabcolsep=5mm
    \begin{tabular}{llllllllll}
    \hline
    & \multicolumn{3}{c}{Relaxation method} &
    & \multicolumn{3}{c}{Direct method} \\
    \cmidrule(r){2-4}\cmidrule(r){6-8}
    \#    & ObjVal  & Infeasibility & Time  &
          & ObjVal  & Infeasibility & Time  \\
    \hline

51  & -3.52         & 11.78     & 461.95  &
    & 5.74e-04      & 182.30    & 28.93   \\

52  & -17.97        & 4.68e-13  & 5.83    &
    & -6.98e-06     & 205.57    & 40.66   \\

53  & 3.21          & 1.28      & 37.06   &
    & 0             & 182.11    & 50.22   \\

54  & 11.38         & 24.13     & 271.87  &
    & 0             & 198.36    & 42.68   \\

55  & -58.85        & 7.87      & 320.34  &
    & 0.02          & 232.94    & 79.18   \\

56  & -37.34        & 0.01      & 52.31   &
    & -0.01         & 199.38    & 14.15   \\
\hline
    \end{tabular}%
  \label{tab:numerical results for MPCC (m=100, p=110)}
\end{table}%

\begin{table}[htbp]
\renewcommand\arraystretch{1}
  \centering
  \caption{Numerical results for MPCC with $\{n=20, l=30, m=160, p=150\}$}
   \tabcolsep=5mm
    \begin{tabular}{llllllll}
    \hline
    & \multicolumn{3}{c}{Relaxation method} &
    & \multicolumn{3}{c}{Direct method} \\
    \cmidrule(r){2-4}\cmidrule(r){6-8}
    \#    & ObjVal  & Infeasibility & Time  &
          & ObjVal  & Infeasibility & Time  \\
    \hline

101 & -20.93       & 1.44      & 355.02  &
    & 0.07         & 325.62    & 222.68  \\

102 & 84.63        & 2.41      & 435.86  &
    & -2.71e-04    & 354.36    & 162.70  \\

103 & -63.89       & 3.95e-05  & 96.02   &
    & -1.70e-06    & 361.92    & 215.61  \\

104 & -57.71       & 3.90      & 766.73  &
    & 0.01         & 337.14    & 174.71  \\

105 & -35.91       & 3.93      & 329.24  &
    & 0.01         & 319.97    & 47.91   \\

106 & 18.52        & 1.74e-12  & 20.15   &
    & 4.76e-04     & 341.99    & 428.19  \\
\hline
    \end{tabular}%
  \label{tab:numerical results for MPCC (m=160, p=150)}
\end{table}%

\begin{table}[htbp]
\renewcommand\arraystretch{1}
  \centering
    \caption{Numerical results for WDP with $\{n=20, l=30, m=60, p=50\}$}
    \tabcolsep=5mm
    \begin{tabular}{llllllll}
    \hline
   & \multicolumn{3}{c}{Relaxation method} &
    & \multicolumn{3}{c}{Direct method} \\
    \cmidrule(r){2-4}\cmidrule(r){6-8}
    \#    & ObjVal & Infeasibility & Time  &
          & ObjVal & Infeasibility & Time  \\
    \hline

1     & 4.06       & 1.84    & 85.42  &
      & 3.92e-03   & 114.20  & 19.16  \\

2     & 6.55       & 0.26    & 19.25  &
      & 0.04       & 130.52  & 7.26  \\

3     & -35.61     & 10.18   & 111.47  &
      & 0.04       & 135.95  & 6.25  \\

4     & 23.20      & 10.93   & 66.11  &
      & 0.09       & 138.84  & 17.60  \\

5     & 13.44      & 0.33    & 9.43  &
      & 3.39e-03   & 148.46  & 4.10  \\

6     & 11.24      & 4.74    & 52.85  &
      & -0.02      & 154.22  & 18.87  \\
\hline
    \end{tabular}%
  \label{tab:numerical results for WDP (m=60, p=50)}
\end{table}%

\begin{table}[htbp]
\renewcommand\arraystretch{1}
  \centering
    \caption{Numerical results for WDP with $\{n=20, l=30, m=100, p=110\}$}
     \tabcolsep=5mm
    \begin{tabular}{llllllllll}
    \hline
    & \multicolumn{3}{c}{Relaxation method} &
    & \multicolumn{3}{c}{Direct method} \\
    \cmidrule(r){2-4}\cmidrule(r){6-8}
    \#    & ObjVal & Infeasibility & Time  &
          & ObjVal & Infeasibility & Time  \\
    \hline

51    & -54.73    & 6.09    & 267.81  &
      & -0.14     & 182.38  & 15.87  \\

52    & -3.66     & 4.64    & 126.34  &
      & 0.11      & 206.01  & 9.35  \\

53    & 57.71     & 1.23e-04& 30.09  &
      & 0         & 182.11  & 29.21  \\

54    & -25.78    & 10.66   & 126.68  &
      & -0.04     & 198.35  & 28.74  \\

55    & -38.89    & 4.89e-03& 67.78  &
      & 0.11      & 233.39  & 12.82  \\

56    & 27.11     & 6.05    & 371.82  &
      & 0.10      & 199.59  & 36.39  \\
\hline
    \end{tabular}%
  \label{tab:numerical results for WDP (m=100, p=110)}
\end{table}%

\begin{table}[htbp]
\renewcommand\arraystretch{1}
  \centering
    \caption{Numerical results for WDP with $\{n=20, l=30, m=160, p=150\}$}
    \tabcolsep=5mm
    \begin{tabular}{llllllll}
    \hline
    & \multicolumn{3}{c}{Relaxation method} &
    & \multicolumn{3}{c}{Direct method} \\
    \cmidrule(r){2-4}\cmidrule(r){6-8}
    \#    & ObjVal  & Infeasibility & Time  &
          & ObjVal  & Infeasibility & Time  \\
    \hline

101   & -65.19     & 2.11     & 909.73  &
      & -0.03      & 325.65   & 35.67  \\

102   & -67.83     & 1.30     & 389.57  &
      & -0.07      & 354.33   & 45.60  \\

103   & 23.00      & 4.70     & 1487.03  &
      & -0.01      & 361.92   & 52.07  \\

104   & -3.05      & 16.86    & 1646.00  &
      & 3.40e-03   & 337.09   & 29.73  \\

105   & 72.06      & 1.55     & 353.81  &
      & 0.02       & 319.93   & 28.44  \\

106   & 18.52      & 1.74e-12 & 35.66  &
      & 0.01       & 341.94   & 45.80  \\
\hline
    \end{tabular}%
  \label{tab:numerical results for WDP (m=160, p=150)}
\end{table}%

\begin{table}[htbp]
\renewcommand\arraystretch{1}
  \centering
    \caption{Numerical results for MDP with $\{n=20, l=30, m=60, p=50\}$}
    \tabcolsep=5mm
    \begin{tabular}{llllllll}
    \hline
   & \multicolumn{3}{c}{Relaxation method} &
    & \multicolumn{3}{c}{Direct method} \\
    \cmidrule(r){2-4}\cmidrule(r){6-8}
    \#    & ObjVal  & Infeasibility & Time  &
          & ObjVal  & Infeasibility & Time  \\
    \hline

1     & 3.21       & 5.24     & 47.29  &
      & 2.68e-03   & 114.22   & 4.62  \\

2     & -16.50     & 0.96     & 18.81  &
      & 0.04       & 130.51   & 7.67  \\

3     & -54.60     & 2.16e-04 & 4.87  &
      & 0.06       & 135.94   & 6.33  \\

4     & -28.70     & 1.01e-08 & 2.90  &
      & 0.09       & 138.90   & 11.58  \\

5     & -1.80      & 7.03e-13 & 1.95  &
      & 1.02e-03   & 148.46   & 4.16  \\

6     & 10.86      & 1.66     & 10.28  &
      & 3.85e-03   & 154.24   & 9.14  \\
\hline
    \end{tabular}%
  \label{tab:numerical results for MDP (m=60, p=50)}
\end{table}%

\begin{table}[htbp]
\renewcommand\arraystretch{1}
  \centering
    \caption{Numerical results for MDP with $\{n=20, l=30, m=100, p=110\}$}
    \tabcolsep=5mm
    \begin{tabular}{llllllll}
    \hline
 & \multicolumn{3}{c}{Relaxation method} &
    & \multicolumn{3}{c}{Direct method} \\
    \cmidrule(r){2-4}\cmidrule(r){6-8}
    \#    & ObjVal  & Infeasibility & Time  &
          & ObjVal  & Infeasibility & Time  \\
    \hline

51    & -55.75      & 8.55e-13   & 13.41  &
      & -0.12       & 182.38     & 14.54  \\

52    & -22.64      & 2.36e-05   & 41.38  &
      & 0.10        & 206.02     & 71.53  \\

53    & -42.35      & 2.12       & 146.52  &
      & 0           & 182.11     & 27.58  \\

54    & 30.74       & 2.11       & 156.74  &
      & -0.05       & 198.35     & 26.76  \\

55    & -5.64       & 7.84e-13   & 8.48  &
      & 0.11        & 233.39     & 13.49  \\

56    & -10.26      & 2.60       & 257.67  &
      & 0.04        & 199.57     & 17.08  \\
\hline
    \end{tabular}%
  \label{tab:numerical results for MDP (m=100, p=110)}
\end{table}%

\begin{table}[htbp]
\renewcommand\arraystretch{1}
  \centering
    \caption{Numerical results for MDP with $\{n=20, l=30, m=160, p=150\}$}
    \tabcolsep=5mm
    \begin{tabular}{llllllll}
    \hline
    & \multicolumn{3}{c}{Relaxation method} &
    & \multicolumn{3}{c}{Direct method} \\
    \cmidrule(r){2-4}\cmidrule(r){6-8}
    \#    & ObjVal  & Infeasibility & Time  &
          & ObjVal  & Infeasibility & Time  \\
    \hline

101   & -2.82      & 2.89      & 500.21  &
      & 0.02       & 325.65    & 54.09  \\

102   & 34.13      & 8.56e-09  & 35.47  &
      & -0.06      & 354.32    & 44.77  \\

103   & -28.62     & 6.74e-04  & 107.01  &
      & 4.21e-03   & 361.92    & 451.83  \\

104   & -22.94     & 2.44      & 792.97  &
      & 0.01       & 337.08    & 331.35  \\

105   & -33.90     & 1.81      & 348.45  &
      & 0.02       & 319.93    & 27.51  \\

106   & -45.69     & 4.17      & 74.42  &
      & 0.01       & 341.95    & 47.16  \\
\hline
    \end{tabular}%
  \label{tab:numerical results for MDP (m=160, p=150)}
\end{table}%

\begin{table}[htbp]
\renewcommand\arraystretch{1}
  \centering
    \caption{Numerical results for eMDP with $\{n=20, l=30, m=60, p=50\}$}
    \tabcolsep=5mm
    \begin{tabular}{llllllll}
    \hline
    & \multicolumn{3}{c}{Relaxation method} &
    & \multicolumn{3}{c}{Direct method} \\
    \cmidrule(r){2-4}\cmidrule(r){6-8}
    \#    & ObjVal  & Infeasibility & Time  &
          & ObjVal  & Infeasibility & Time  \\
    \hline
1     & -18.72    & 4.60e-13   & 2.26     &
      & 0.12      & 114.19     & 4.71  \\

2     & -6.78     & 1.54       & 51.06  &
      & 0.04      & 130.51     & 27.72  \\

3     & 28.56     & 8.02e-09   & 3.33  &
      & 0.09      & 135.91     & 11.05  \\

4     & -28.70    & 9.50e-09   & 3.25  &
      & -0.15     & 138.84     & 60.82  \\

5     & -1.80     & 7.03e-13   & 2.38  &
      & 1.94e-03  & 148.46     & 3.86  \\

6     & 23.83     & 0.03       & 158.95  &
      & 0.16      & 154.28     & 4.50  \\
\hline
    \end{tabular}%
  \label{tab:numerical results for eMDP (m=60, p=50)}
\end{table}%

\begin{table}[htbp]
\renewcommand\arraystretch{1}
  \centering
    \caption{Numerical results for eMDP with $\{n=20, l=30, m=100, p=110\}$}
    \tabcolsep=5mm
    \begin{tabular}{llllllll}
    \hline
    & \multicolumn{3}{c}{Relaxation method} &
    & \multicolumn{3}{c}{Direct method} \\
    \cmidrule(r){2-4}\cmidrule(r){6-8}
    \#    & ObjVal  & Infeasibility & Time  &
          & ObjVal  & Infeasibility & Time  \\
    \hline
51    & -43.66     & 1.02e-08     & 12.88  &
      & 0.06       & 182.65       & 212.89  \\

52    & -17.97     & 4.68e-13     & 10.82  &
      & 0.05       & 205.92       & 236.19  \\

53    & -43.21     & 1.07e-12     & 11.24  &
      & 0          & 182.11       & 48.04  \\

54    & 40.13      & 1.07e-08     & 13.08  &
      & -0.05      & 198.35       & 35.85  \\

55    & 38.10      & 3.39e-05     & 514.86  &
      & -0.09      & 233.88       & 420.26  \\

56    & -5.86      & 5.99e-13     & 10.23  &
      & -0.03      & 199.57       & 56.23  \\
\hline
    \end{tabular}%
  \label{tab:numerical results for eMDP (m=100, p=110)}
\end{table}%

\begin{table}[htbp]
\renewcommand\arraystretch{1}
  \centering
    \caption{Numerical results for eMDP with $\{n=20, l=30, m=160, p=150\}$}
    \tabcolsep=5mm
    \begin{tabular}{llllllll}
    \hline
    & \multicolumn{3}{c}{Relaxation method} &
    & \multicolumn{3}{c}{Direct method} \\
    \cmidrule(r){2-4}\cmidrule(r){6-8}
    \#    & ObjVal  & Infeasibility & Time  &
          & ObjVal  & Infeasibility & Time  \\
    \hline
101   & -23.26     & 1.03e-08     & 46.34  &
      & -0.03      & 325.65       & 758.89  \\

102   & 34.13      & 7.17e-09     & 43.42  &
      & -0.05      & 354.33       & 898.23  \\

103   & 4.56       & 1.71e-12     & 38.53  &
      & 0.01       & 361.89       & 1416.60  \\

104   & -5.69      & 1.06e-08     & 42.54  &
      & 0.03       & 337.10       & 855.68  \\

105   & 37.74      & 1.06e-08     & 42.28  &
      & 0.02       & 319.92       & 1041.51  \\

106   & 18.52      & 1.74e-12     & 39.48  &
      & -0.02      & 341.90       & 1813.56  \\
\hline
    \end{tabular}%
  \label{tab:numerical results for eMDP (m=160, p=150)}
\end{table}%

\end{appendices}

%
%
%
%
%
%

	\end{document}